\documentclass[12pt,a4paper]{article}
\usepackage{latexsym,amssymb,amsfonts,amsmath,amsthm,nccmath}
\usepackage[hidelinks]{hyperref}
\usepackage[margin=2cm]{geometry}
\usepackage[usenames,dvipsnames]{xcolor}

\numberwithin{equation}{section}
\newtheorem{theorem}{Theorem}
\newtheorem{corollary}[theorem]{Corollary}
\newtheorem{lemma}[theorem]{Lemma}

\theoremstyle{definition}
\newtheorem*{definition}{Definition}

\def \deg {{\rm deg}}

\def \D {{\mathcal{D}}}

\def \P {{\mathcal {P}}}
\def \a {\alpha}
\def \b {\beta}

\def \T {\tau}

\def \mod {\ ({\rm mod}\,\,}

\def \deg {{\rm deg}}
\def \N {{\rm Nbd}}
\def \binom#1#2{{#1\choose#2}}

\def \mod#1{{\:({\rm mod}\ #1)}}

\let\oldproofname=\proofname
\renewcommand{\proofname}{\rm\bf{\oldproofname}}

\title{\bf Doyen-Wilson results for odd length cycle systems}

\author{
Daniel Horsley and Rosalind A. Hoyte\\
School of Mathematical Sciences \\
Monash University \\
Vic 3800, Australia \\[0.1cm]
\texttt{danhorsley@gmail.com}, \texttt{rosalind.hoyte@monash.edu}}
\date{}

\begin{document}
\maketitle
\def\baselinestretch{1.2}\small\normalsize
\sloppy

\begin{abstract}
For each odd $m \geq 3$ we completely solve the problem of when an $m$-cycle system of order $u$ can be embedded in an $m$-cycle system of order $v$, barring a finite number of possible exceptions. In cases where $u$ is large compared to $m$, where $m$ is a prime power, or where $m \leq 15$, the problem is completely resolved. In other cases, the only possible exceptions occur when $v-u$ is small compared to $m$. This result is proved as a consequence of a more general result which gives necessary and sufficient conditions for the existence of an $m$-cycle decomposition of a complete graph of order $v$ with a hole of size $u$ in the case where $u \geq m-2$ and $v-u \geq m+1$ both hold.
\end{abstract}

\section{Introduction}

An \emph{$m$-cycle decomposition} of a graph $G$ is a collection of cycles of length $m$ in $G$ whose edge sets form a partition of the edge set of $G$. An \emph{$m$-cycle system of order $v$} is an $m$-cycle decomposition of the complete graph of order $v$. Cycle systems of order one exist trivially. Building on work of Hoffman, Lindner and Rodger \cite{HofLinRod89}, Alspach and Gavlas \cite{AlsGav01} and \v{S}ajna \cite{Sajna02} established that the obvious necessary conditions for the existence of an $m$-cycle system of order $v$ were also sufficient.

\begin{theorem}[\cite{AlsGav01,Sajna02}]\label{CSExist}
Let $m \geq 3$ and $v > 1$ be integers. There exists an $m$-cycle system of order $v$ if and only if $v$ is odd, $v \geq m$, and $\binom{v}{2} \equiv 0 \mod{m}$.
\end{theorem}

An $m$-cycle system $\mathcal{A}$ is said to be embedded in another $m$-cycle system $\mathcal{B}$ when $\mathcal{A} \subseteq \mathcal{B}$. Every $m$-cycle system can be trivially embedded in itself. The problem of determining when an $m$-cycle system of order $u$ can be embedded in an $m$-cycle system of order $v$ has been well studied, although it remains open. The problem takes on quite a different complexion depending on whether $m$ is odd or even. Our focus here will be on the case where $m$ is odd. Famously, the problem was solved in the case of $3$-cycles by Doyen and Wilson \cite{DoyWil73}. Subsequently, this result was extended to the case of $m$-cycles for $m=5$ \cite{BryRod94_5}, then to $m \in \{7,9\}$ \cite{BryRod94_m}, and finally to $m \in \{11,13\}$ \cite{BryRodSpi97}. Also, in \cite{BryRod94_m} the problem was solved in the case where $u,v \equiv 1 \mbox{ or } m \mod{2m}$, barring at most one exception for each $m$ and $u$.

Here we completely solve the problem in the case when $u > \frac{(m-1)(m-2)}{2}$ and, for other values of $m$ and $u$, we solve it apart from cases where $u < v \leq u+m-1$. This means that the problem is solved for each $m$ with the exception of finitely many possible cases. We completely solve the problem in the case where $m$ is a prime power by resolving the possible exceptions in a result in \cite{BryRod94_m}.

\begin{theorem}\label{DWTheorem}
Let $m\geq 3$ be an odd integer and let $u$ and $v$ be positive integers with $u<v$.
\begin{itemize}
    \item[(i)]
If $u > \frac{(m-1)(m-2)}{2}$ or if $m$ is a prime power, then an $m$-cycle system of order $u$ can be embedded in an $m$-cycle system of order $v$ if and only if $u$ and $v$ are odd, $\binom{u}{2},\binom{v}{2} \equiv 0 \mod{m}$ and $v \geq \frac{u(m+1)}{m-1}+1$.
    \item[(ii)]
If $u \leq \frac{(m-1)(m-2)}{2}$ and $m$ is not a prime power, then an $m$-cycle system of order $u$ can be embedded in an $m$-cycle system of order $v$ if and only if $u$ and $v$ are odd, $\binom{u}{2},\binom{v}{2} \equiv 0 \mod{m}$ and $v \geq \frac{u(m+1)}{m-1}+1$, except that the embedding may not exist when $\frac{u(m+1)}{m-1}+1 \leq v \leq u+m-1$.
\end{itemize}
\end{theorem}

In the above it is easy to see that the conditions that $u$ and $v$ are odd and $\binom{u}{2},\binom{v}{2} \equiv 0 \mod{m}$ are necessary for the embedding to exist. The condition $v \geq \frac{u(m+1)}{m-1}+1$ can also be seen to be necessary by observing that each $m$-cycle which is not part of the original system must contain two consecutive vertices not in the original system (see Lemma \ref{NecConds} below).

In fact, we prove a result more general than Theorem \ref{DWTheorem} concerning $m$-cycle decompositions of complete graphs with holes. For positive integers $u$ and $v$ with $u<v$, the \emph{complete graph of order $v$ with a hole of size $u$}, denoted $K_v-K_u$, is the graph obtained from a complete graph of order $v$ by removing the edges of a complete subgraph of order $u$. Any embedding of an $m$-cycle system of order $u$ in another of order $v$ yields an $m$-cycle decomposition of $K_v-K_u$ (via removing the cycles in the original system), but the problem of finding $m$-cycle decompositions of complete graphs with holes is more general because the orders of the graph and hole need not be feasible orders for $m$-cycle systems. This more general problem has also received significant attention. For odd $m$, the problem is completely solved for $m=3$ \cite{MenRos83}, $m=5$ \cite{BryHofRod96}, and $m=7$ \cite{BryRodSpi97}. For a survey of results concerning cycle decompositions, see \cite{BryRod07}.
The following lemma from \cite{BryRodSpi97} gives well-known necessary conditions for the existence of an $m$-cycle decomposition of $K_v-K_u$.

\begin{lemma}[\cite{BryRodSpi97}]\label{NecConds}
Let $m \geq 3$ be an odd integer and let $u$ and $v$ be positive integers such that $v>u$. If there exists an $m$-cycle decomposition of $K_v-K_u$, then
\begin{itemize}
    \item[$(\rm{N}1)$]
$u$ and $v$ are odd;
    \item[$(\rm{N}2)$]
$\binom{v}{2}-\binom{u}{2} \equiv 0 \mod{m}$;
    \item[$(\rm{N}3)$]
$v \geq \frac{u(m+1)}{m-1}+1$; and
    \item[$(\rm{N}4)$]
$(v-m)(v-1) \geq u(u-1)$.
\end{itemize}
\end{lemma}

The necessity of (N3) follows from the fact that each $m$-cycle must contain two consecutive vertices outside the hole. The necessity of (N4) follows from the fact that $\frac{v-1}{2}$ cycles of the decomposition contain some vertex outside the hole, and hence $K_v-K_u$ must have at least $\frac{m(v-1)}{2}$ edges. Condition (N4) implies that $v \geq u+\frac{m}{2}$. It will be important for our purposes that (N4) is always satisfied if $v \geq u+m-1$. For an odd integer $m \geq 3$, we say that a pair $(u,v)$ of positive integers is
\emph{$m$-admissible} if $u$ and $v$ satisfy conditions (N1)--(N4). Our main result shows that these necessary conditions are sufficient whenever $u \geq m-2$ and $v-u \geq m+1$.

\begin{theorem}\label{MainTheorem}
Let $m\geq 3$ be an odd integer and let $u$ and $v$ be integers such that $u \geq m-2$ and $v-u \geq m+1$. There exists an $m$-cycle decomposition of $K_v-K_u$ if and only if
\begin{itemize}
    \item[(i)]
$u$ and $v$ are odd;
    \item[(ii)]
$\binom{v}{2}-\binom{u}{2} \equiv 0 \mod{m}$; and
    \item[(iii)]
$v \geq \frac{u(m+1)}{m-1}+1$.
\end{itemize}
\end{theorem}

Theorem \ref{MainTheorem} complements a similar result for cycles of fixed even length; see Theorem 1.2 of \cite{Horsley12}. As a consequence of Theorem \ref{MainTheorem} we find the following.

\begin{corollary}\label{MissingCasesTheorem}
Let $m$ and $u$ be odd integers and let $\nu_m(u)$ be the smallest integer $x>u$ such that $(u,x)$ is $m$-admissible.
\begin{itemize}
	\item[(i)]
If $3<u<m-2$ and there exists an $m$-cycle decomposition of $K_{v'}-K_u$ for each integer $v'$ such that $(u,v')$ is $m$-admissible and $\nu_m(u) \leq v' \leq \nu_m(u)+m-1$, then there exists an $m$-cycle decomposition of $K_v-K_u$ if and only if $(u,v)$ is $m$-admissible.
	\item[(ii)]
If $m-2\leq u\leq \frac{(m-1)(m-2)}{2}$, then there exists an $m$-cycle decomposition of $K_v-K_u$ if and only if $(u,v)$ is $m$-admissible, except that this decomposition may not exist when $\nu_m(u) \leq v \leq u+m-1$.
	\item[(iii)]
If $u> \frac{(m-1)(m-2)}{2}$ or $u\in \{1,3\}$, then there exists an $m$-cycle decomposition of $K_v-K_u$ if and only if $(u,v)$ is $m$-admissible.
\end{itemize}
\end{corollary}

Note that $\nu_m(u)$ in the above corollary is at most the smallest integer $y \equiv u \mod{2m}$ such that $y \geq \frac{u(m+1)}{m-1}+1$, because $(u,y)$ is $m$-admissible for any such integer. Corollary \ref{MissingCasesTheorem} makes it clear that for a given odd $m$, we can establish the existence of an $m$-cycle decomposition of ${K_v-K_u}$ for all $m$-admissible $(u,v)$ provided we can construct a number of ``small'' decompositions. We have been able to do this for $m \in \{9,11,13,15\}$ and thus have resolved the problem for each odd $m \leq 15$.

\begin{theorem}\label{SmallMSolution}
Let $m$ be an odd integer such that $3 \leq m \leq 15$ and let $u$ and $v$ be positive integers such that $v>u$. Then there exists an $m$-cycle decomposition of $K_v-K_u$ if and only if $(u,v)$ is $m$-admissible.
\end{theorem}

Theorem \ref{MainTheorem} is proved by beginning with a cycle decomposition of $K_v-K_u$ that involves many short cycles and iteratively altering our decomposition of $K_v-K_u$ so as to ``merge'' a number of short cycle lengths until we eventually obtain an $m$-cycle decomposition of $K_v-K_u$. We alter the decompositions using ``cycle switching'' techniques first developed in \cite{BryHorMae05}. Section \ref{Notation} is devoted to introducing the notation and definitions that we will require, as well as the fundamental lemma encapsulating our cycle switching techniques. Sections \ref{2ChainSec}--\ref{JoiningSec} are devoted to proving Lemma \ref{JoiningLemma} which shows that we can alter a cycle decomposition of $K_v-K_u$ in the required manner. In Section \ref{BaseDecompSec} we construct the decompositions of $K_v-K_u$ involving short cycles that are required as a ``base'' for our construction. Finally in Section \ref{ProofSec} we combine these results in order to prove Theorem \ref{MainTheorem} and its consequences.

\section{Notation and Preliminary Results}\label{Notation}

A \emph{decomposition} of a graph $G$ is a collection of subgraphs of $G$ whose edges form a partition of the edge set of $G$. A \emph{packing} of a graph $G$ is a decomposition of some subgraph $H$ of $G$ and the \emph{leave} of the packing is the graph obtained by removing the edges of $H$ from $G$. We define the \emph{reduced leave} of a packing of a graph $G$ as the graph obtained from its leave by deleting any isolated vertices. For a list of positive integers $M=m_1,\ldots,m_t$, an \emph{$(M)$-decomposition} of a graph $G$ is a decomposition of $G$ into $t$ cycles of lengths $m_1,m_2,\dots,m_t$ and an \emph{$(M)$-packing} of $G$ is a packing of $G$ with $t$ cycles of lengths $m_1,m_2,\dots,m_t$. All lists in this paper will be lists of positive integers.

The $m$-cycle with vertices $x_0,x_1,\ldots,x_{m-1}$ and edges $x_ix_{i+1}$ for $i\in\{0,\ldots,m-1\}$ (with subscripts modulo $m$) is denoted by $(x_0,x_1,\ldots,x_{m-1})$ and the $n$-path with vertices $y_0,y_1,\ldots,y_n$ and edges $y_jy_{j+1}$ for $j\in \{0,1,\ldots,n-1\}$ is denoted by $[y_0,y_1,\ldots,y_n]$. We will say that $y_0$ and $y_n$ are the end vertices of this path.

For a positive integer $v$, let $K_v$ denote the complete graph of order $v$, and for a set $V$, let $K_V$ denote the complete graph with vertex set $V$. For positive integers $u$ and $w$, let $K_{u,w}$ denote the complete bipartite graph with parts of size $u$ and $w$, and for disjoint sets $U$ and $W$, let $K_{U,W}$ denote the complete bipartite graph with parts $U$ and $W$. For graphs $G$ and $H$, we denote by $G \cup H$ the graph with vertex set $V(G) \cup V(H)$ and edge set $E(G) \cup E(H)$, we denote by $G-H$ the graph with vertex set $V(G)$ and edge set $E(G) \setminus E(H)$, and, if $V(G)$ and $V(H)$ are disjoint, we denote by $G \vee H$ the
graph with vertex set $V(G)\cup V(H)$ and edge set $E(G)\cup E(H)\cup E(K_{V(G),V(H)})$ (our use of this last notation will imply that $V(G)$ and $V(H)$ are disjoint).

The neighbourhood $\N_G(x)$ of a vertex $x$ in a graph $G$ is the set of vertices in $G$ that are adjacent to $x$ (not including $x$ itself). We say vertices $x$ and $y$ of a graph $G$ are \textit{twin in $G$} if $\N_G(x)\setminus\{y\} =\N_G(y)\setminus\{x\}$. Let $U$ and $V$ be sets with $U \subseteq V$, and let $G$ be the graph $K_V-K_U$. Note that the vertices in $U$ are pairwise twin and the vertices in $V \setminus U$ are pairwise twin. We say an edge $xy$ of $G$ is a \emph{pure edge} if $x,y\in V \setminus U$, and we say that it is a \emph{cross edge} if $x$ or $y\in U$.

Given a permutation $\pi$ of a set $V$, a subset $S$ of $V$ and a graph $G$ with $V(G)\subseteq V$, $\pi(S)$ is defined to be the set $\{\pi(x):x \in S\}$ and $\pi(G)$ is defined to be the graph with vertex set $\{\pi(x):x\in V(G)\}$ and edge set $\{\pi(x)\pi(y):xy\in E(G)\}$.

The following lemma encapsulates the switching technique that is fundamental to the results in Sections \ref{2ChainSec}--\ref{JoiningSec}. It is almost identical to Lemma 2.1 of \cite{Horsley12} and the proof given in that paper suffices to prove this result as well.

	
\begin{lemma}\label{CycleSwitchGeneral}
Let $u$ and $v$ be positive odd integers with $u\leq v$, and let $M$ be a list of integers. Let $\mathcal{P}$ be an $(M)$-packing of $K_v-K_u$ with leave $L$,  let $\a$ and $\b$ be twin
vertices in $K_v-K_u$, and let $\pi$ be the transposition $(\a\b)$. Then there exists a partition of the set $(\N_L(\a)\cup \N_L(\b)) \setminus
((\N_L(\a)\cap \N_L(\b)) \cup \{\a,\b\})$ into pairs such that for each pair
$\{x,y\}$ of the partition, there exists an $(M)$-packing $\mathcal{P}'$ of $K_v-K_u$ whose leave
$L'$ differs from $L$ only in that each of $\a x$, $\a y$, $\b x$ and $\b y$ is an edge in $L'$ if and only if it is not an edge in $L$. Furthermore, if $\mathcal{P} = \{C_1,C_2,\ldots,C_t\}$, then $\mathcal{P}' = \{C'_1,C'_2,\ldots,C'_t\}$, where, for each $i \in \{1,\ldots,t\}$, $C'_i$ is a cycle of the same length as $C_i$ such that
\begin{itemize}
    \item[(i)]
if neither $\alpha$ nor $\beta$ is in $V(C_i)$, then $C'_i=C_i$;
    \item[(ii)]
if exactly one of $\alpha$ and $\beta$ is in $V(C_i)$, then either
$C'_i=C_i$ or $C'_i=\pi(C_i)$; and
    \item[(iii)]
if both $\alpha$ and $\beta$ are in $V(C_i)$, then $C'_i \in
\{C_i,\pi(C_i),\pi(P_i)\cup P^{\dag}_i,P_i \cup \pi(P^{\dag}_i)\}$, where $P_i$ and
$P^{\dag}_i$ are the two paths in $C_i$ which have end vertices $\alpha$ and $\beta$.
\end{itemize}
\end{lemma}

If we are applying Lemma~\ref{CycleSwitchGeneral} we say that we are performing the $(\a,\b)$-switch with origin $x$ and terminus $y$ (equivalently, with origin $y$ and terminus $x$). Note that if $U$ and $V$ are sets with $U \subseteq V$, then two vertices $\a,\b\in V$ are twin in $K_V-K_U$ if and only if $\{\a,\b\}\subseteq U$ or $\{\a,\b\}\subseteq V \setminus U$.

\begin{definition}
Let $G$ be a graph, and let $\mathcal{P} = \{G_1,\ldots,G_t\}$ be a packing of $G$. We say that another packing $\mathcal{P}'$ of $G$ is a \emph{repacking} of $\mathcal{P}$ if $\mathcal{P}' = \{G'_1,\ldots,G'_t\}$ where for each $i \in \{1,\ldots,t\}$ there is a permutation $\pi_i$ of $V(G)$ such that $\pi_i(G_i)=G'_i$ and $x$ and $\pi_i(x)$ are twin in $G$ for each $x \in V(G)$.
\end{definition}

Obviously, for any list of integers $M$, any repacking of an $(M)$-packing of a graph $G$ is also an $(M)$-packing of $G$. If $G$ is a complete graph with a hole, then the above definition implies that $G_i$ and $G'_i$ have the same number of pure and cross edges for each $i \in \{1,\ldots,t\}$ and hence also that the leaves of $\mathcal{P}$ and $\mathcal{P}'$ have the same number of pure and cross edges. If $\mathcal{P}$ is a packing of a graph $G$, $\mathcal{P}'$ is a repacking of $\mathcal{P}$ and $\mathcal{P}''$ is a repacking of $\mathcal{P}'$, then $\mathcal{P}''$ is obviously also a repacking of $\mathcal{P}$. If $\mathcal{P}$ is a packing of a graph $G$ and $\mathcal{P}'$ is another packing of $G$ obtained from $\mathcal{P}$ by applying Lemma \ref{CycleSwitchGeneral}, then $\mathcal{P}'$ is necessarily a repacking of $\mathcal{P}$.

The following definitions for two types of graphs, rings and chains, are the same as in \cite{Horsley12}.
\begin{definition}
An $(a_1,a_2,\ldots,a_s)$-chain (or $s$-chain if we do not wish to specify the lengths of the cycles) is the edge-disjoint union of $s\geq 2$ cycles $A_1,A_2,\ldots, A_s$ such that
\begin{itemize}
\item $A_i$ is a cycle of length $a_i$ for $1\leq i\leq s$; and
\item for $1\leq i<j\leq s$, $|V(A_i)\cap V(A_j)|=1$ if $j=i+1$ and $|V(A_i)\cap V(A_j)|=0$ otherwise.
\end{itemize}
\end{definition}
We call $A_1$ and $A_s$ the \emph{end cycles} of the chain, and for $1<i<s$ we call $A_i$ an \emph{internal cycle} of the chain. A vertex which is in two cycles of the chain is said to be the \emph{link vertex} of those cycles. We denote a $2$-chain with cycles $P$ and $Q$ by $P\cdot Q$.

\begin{definition}
An $(a_1,a_2,\ldots,a_s)$-ring (or $s$-ring if we do not wish to specify the lengths of the cycles) is the edge-disjoint union of $s\geq 2$ cycles $A_1,A_2,\ldots, A_s$ such that
\begin{itemize}
\item $A_i$ is a cycle of length $a_i$ for $1\leq i\leq s$;
\item for $s \geq 3$ and $1\leq i<j\leq s$, $|V(A_i)\cap V(A_j)|=1$ if $j=i+1$ or if $(i,j)=(1,s)$, and $|V(A_i)\cap V(A_j)|=0$ otherwise; and
\item if $s=2$ then $|V(A_1)\cap V(A_2)|=2$.
\end{itemize}
\end{definition}
We refer to the cycles $A_1,A_2,\ldots,A_s$ as the \emph{ring cycles} of the ring in order to distinguish them from the other cycles that can be found within the graph. A vertex which is in two ring cycles of the ring is said to be a \emph{link vertex} of those cycles.

\section{Packings whose leaves are 2-chains}\label{2ChainSec}
Our aim in Sections~\ref{2ChainSec}--\ref{JoiningSec} is to prove Lemma~\ref{JoiningLemma}. This lemma allows us to begin with a packing of $K_v-K_u$ satisfying various conditions and find a repacking whose leave can be decomposed into two $m$-cycles, each with exactly one pure edge. Finding $m$-cycles of this form is important because, in an $m$-cycle decomposition of $K_v-K_u$ with $m$ odd and $v = \frac{u(m+1)}{m-1}+1$ (that is, with equality in necessary condition (N3)), every cycle must contain exactly one pure edge. Thus in Lemmas~\ref{2PureEdgesOddCycles}--\ref{JoiningLemma} we focus on packings of $K_v-K_u$ whose leaves have exactly two pure edges (recall that repacking preserves the number of pure and cross edges in the leave).

In this section we focus on starting with a packing whose reduced leave is a $2$-chain and finding a repacking whose reduced leave is the edge-disjoint union of two cycles of specified lengths. Our main goal here is to prove Lemma \ref{2PureEdges}. The other lemmas in this section are used only in order to prove it. Lemmas \ref{FigureOfEight1}--\ref{Rearrange_Local} apply to packings of arbitrary graphs, while in Lemmas \ref{2PureEdgesOddCycles}--\ref{2PureEdges} we concentrate on packings of complete graphs with holes whose leaves have exactly two pure edges.

\begin{lemma}\label{FigureOfEight1}
Let $G$ be a graph and let $M$ be a list of integers. Let $m$, $p$ and $q$ be positive integers with $m\geq p$ and $p+q-m\geq 3$. Suppose there exists an $(M)$-packing $\mathcal{P}$ of $G$ whose reduced leave is a $(p,q)$-chain $(x_1, x_2, \ldots, x_{p-1},c)\cdot(c, y_1, y_2, \ldots, y_{q-1})$ such that $x_1$ and $y_{m-p+1}$ are twin in $G$. Then there exists a repacking of $\mathcal{P}$ whose reduced leave is either
\begin{itemize}
    \item
the edge-disjoint union of an $m$-cycle and a $(p+q-m)$-cycle; or
    \item
the $(m-p+2, 2p+q-m-2)$-chain given by $(x_1,y_{m-p},y_{m-p-1},\ldots,y_1,c)\cdot(c,x_{p-1},x_{p-2},\ldots,x_2,y_{m-p+1},y_{m-p+2},\ldots,y_{q-1})$.
\end{itemize}
\end{lemma}

\begin{proof}
Note that $p+q-m \geq 3$ implies that $m-p+1 \leq q-2$. If $p=m$ then we are finished, so assume $p<m$. Since $x_1$ and $y_{m-p+1}$ are twin in $G$, we can perform the $(x_1, y_{m-p+1})$-switch with origin $x_2$. If the switch has terminus $y_{m-p}$, then we obtain a repacking of $\mathcal{P}$ whose reduced leave is the $(m-p+2, 2p+q-m-2)$-chain $(x_1,y_{m-p},y_{m-p-1},\ldots,y_1,c) \cdot (c,x_{p-1},x_{p-2},\ldots,x_2,y_{m-p+1},y_{m-p+2},\ldots,y_{q-1})$. Otherwise the switch has terminus $c$ or $y_{m-p+2}$ and in either case we obtain a repacking of $\mathcal{P}$ whose reduced leave is the edge-disjoint union of the $m$-cycle $(y_1, y_2,\ldots,y_{m-p+1}, x_2, x_3, \ldots, x_{p-1}, c)$  and some $(p+q-m)$-cycle.
\end{proof}

\begin{lemma}\label{FigureOfEight2}
Let $G$ be a graph and let $M$ be a list of integers. Let $m$, $p$ and $q$ be positive integers with $m\geq p$ and $p+q-m\geq 3$. Suppose there exists an $(M)$-packing $\mathcal{P}$ of $G$ whose reduced leave is a $(p,q)$-chain $(x_1, x_2, \ldots, x_{p-1},c)\cdot(c, y_1, y_2, \ldots, y_{q-1})$ such that $x_2$ and $y_{m-p+2}$ are twin in $G$. Then there exists a repacking of $\mathcal{P}$ whose reduced leave is either
\begin{itemize}
    \item
the edge-disjoint union of an $m$-cycle and a $(p+q-m)$-cycle; or
    \item
the $(m-p+4, 2p+q-m-4)$-chain given by $(x_1,x_2,y_{m-p+1},y_{m-p},\ldots,y_1,c) \cdot (c,x_{p-1},x_{p-2},\ldots,x_3,y_{m-p+2},y_{m-p+3},\ldots,y_{q-1})$.
\end{itemize}
\end{lemma}

\begin{proof} Note that $p+q-m \geq 3$ implies that $m-p+2 \leq q-1$. If $p=m$ then we are finished, so assume $p<m$. Since $x_2$ and $y_{m-p+2}$ are twin in $G$, we can perform the $(x_2, y_{m-p+2})$-switch with origin $x_3$. If the switch has terminus $y_{m-p+1}$, then we obtain a repacking of $\mathcal{P}$ whose reduced leave is the $(m-p+4, 2p+q-m-4)$-chain $(x_1, x_2,y_{m-p+1},y_{m-p},\ldots,y_1,c) \cdot (c,x_{p-1},x_{p-2},\ldots,x_3,y_{m-p+2},y_{m-p+3},\ldots,y_{q-1})$. Otherwise the switch has terminus $y_{m-p+3}$ or $x_1$ and in either case we obtain a repacking of $\mathcal{P}$ whose reduced leave is the edge-disjoint union of the $m$-cycle $(y_1,y_2,\ldots,y_{m-p+2},x_3,x_4,\ldots,x_{p-1},c)$ and some $(p+q-m)$-cycle.
\end{proof}

\begin{lemma}\label{FigureOfEight3}
Let $G$ be a graph and let $M$ be a list of integers. Let $m$, $p$ and $q$ be positive integers with $m$ odd, $m\geq p$ and $p+q-m\geq 3$. Suppose there exists an $(M)$-packing $\mathcal{P}$ of $G$ whose reduced leave is a $(p,q)$-chain $(x_1,x_2,\ldots,x_{p-1},c) \cdot (c,y_1,y_2,\ldots,y_{q-1})$ such that either
\begin{itemize}
    \item[(i)]
$p$ is odd, $x_1, y_3, y_5,\ldots, y_{m-p+1}$ are pairwise twin in $G$ and $y_2, y_4,\ldots, y_{m-p+2}$ are pairwise twin in $G$; or
    \item[(ii)]
$p$ is even, $x_1,x_3,\ldots, x_{p-3}$ are pairwise twin in $G$ and $y_{m-p+2},x_2,x_4,\ldots,x_{p-2}$ are pairwise twin in $G$.
\end{itemize}
Then there exists a repacking of $\mathcal{P}$ whose reduced leave is the edge-disjoint union of an $m$-cycle and a $(p+q-m)$-cycle.
\end{lemma}

\begin{proof}
If $p=m$, then we are finished. If $p=4$, then $x_2$ and $y_{m-2}$ are twin in $G$ and we can apply Lemma~\ref{FigureOfEight2} to obtain the required packing. So we may assume $p \notin \{4,m\}$.
Let $p_0, p_1,\ldots,p_{\ell}$ be the sequence $m, 4, m-2, 6,\ldots, 7, m-3, 5, m-1, 3.$ For some $k \in \{2,\ldots,\ell\}$ assume that the lemma holds for $p=p_{k-1}$. We will now show that it holds for $p=p_k$.

\noindent{\bf Case 1.} Suppose $p=p_k$ is odd. Since $x_1$ and $y_{m-p+1}$ are twin in $G$, Lemma~\ref{FigureOfEight1} can be applied to obtain a repacking $\mathcal{P}'$ of $\mathcal{P}$. Either we are finished, or the reduced leave of $\mathcal{P}'$ is a $(p',q')$-chain where $p'=m-p+2$ and $q'=2p+q-m-2$. We give this chain and, below it, a relabelling of its vertices.
$$\arraycolsep=0pt%
\begin{array}{cccccccccccccccccc}
  (&x_1, \:&\: y_{m-p}, \:&\: y_{m-p-1}, \:&\: \ldots, \:&\: y_1, \:&\: c&)  \cdot  (&c, &\: x_{p-1}, \:&\: x_{p-2}, \:&\: \ldots, \:&\: x_2, \:&\: y_{m-p+1}, \:&\: y_{m-p+2}, \:&\: \ldots, \:&\: y_{q-1}&) \\
  (&\:x'_1, \:&\: x'_2, \:&\: x'_3, \:&\: \ldots, \:&\: x'_{p'-1}, \:&\: c &)  \cdot (& c, \:&\: y'_1, \:&\: y'_2, \:&\: \ldots, \:&\: y'_{p-2}, \:&\: y'_{p-1}, \:&\: y'_p, \:&\: \ldots, \:&\: y'_{q'-1}\:&)
\end{array}$$
Note that $p'=p_{k-1}$ and $p'$ is even.  Since $x'_1=x_1$ and $\{x'_3,x'_5,\ldots, x'_{p'-3}\} = \{y_3,y_5,\ldots,y_{m-p-1}\}$, the vertices $x'_1,x'_3,\ldots,x'_{p'-3}$ are pairwise twin in $G$.
Similarly, since $y'_{m-p'+2} = y'_p = y_{m-p+2}$ and $\{x'_2,x'_4,\ldots,x'_{p'-2}\} =\{y_2,y_4,\ldots, y_{m-p}\}$, the vertices $y'_{m-p'+2}, x'_2,x'_4,\ldots,x'_{p'-2}$ are pairwise twin in $G$. Thus $\mathcal{P}'$ satisfies (ii) and we are finished by our inductive hypothesis.

\noindent{\bf Case 2.} Suppose $p=p_k$ is even. Then, since $x_2$ and $y_{m-p+2}$ are twin in $G$, Lemma~\ref{FigureOfEight2} can be applied to obtain a repacking $\mathcal{P}'$ of $\mathcal{P}$. Either we are finished, or the reduced leave of $\mathcal{P}'$ is a $(p',q')$-chain where $p'=m-p+4$ and $q'=2p+q-m-4$. We give this chain and, below it, a relabelling of its vertices.
$$\arraycolsep=0pt%
\begin{array}{ccccccccccccccccccc}
  (&x_1, \:&\: x_2, \:&\:y_{m-p+1}, \:&\: y_{m-p}, \:&\: \ldots, \:&\: y_1, \:&\: c&)  \cdot  (&c, &\: x_{p-1}, \:&\: x_{p-2}, \:&\: \ldots, \:&\: x_3, \:&\: y_{m-p+2}, \:&\: y_{m-p+3}, \:&\: \ldots, \:&\: y_{q-1}&) \\
  (&\:x'_1, \:&\: x'_2, \:&\: x'_3, \:&\: x'_4, \:&\: \ldots, \:&\: x'_{p'-1}, \:&\: c &)  \cdot (& c, \:&\: y'_1, \:&\: y'_2, \:&\: \ldots, \:&\: y'_{p-3}, \:&\: y'_{p-2}, \:&\: y'_{p-1}, \:&\: \ldots, \:&\: y'_{q'-1}\:&)
\end{array}$$
Note that $p'=p_{k-1}$ and $p'$ is odd. Since $x'_1=x_1$ and $\{y'_3,y'_5,\ldots, y'_{m-p'+1} \} =\{x_3,x_5,\ldots, x_{p-3}\} $, the vertices $x'_1,y'_3,y'_5,\ldots,y'_{m-p'+1}$ are pairwise twin in $G$. Similarly, since
$\{y'_2,y'_4,\ldots,y'_{m-p'+2} \} =\{x_4,x_6,\ldots,x_{p-2}\}\cup\{y_{m-p+2}\}$, the vertices $y'_2,y'_4,\ldots,y'_{m-p'+2}$ are pairwise twin in $G$. Thus $\mathcal{P}'$ satisfies (i) and we are finished by our inductive hypothesis.
\end{proof}


\begin{lemma}\label{Rearrange_Local}
Let $G$ be a graph and let $M$ be a list of integers. Let $m$, $p$ and $q$ be positive integers with $m$ odd, $m\geq p$, $p+q-m\geq 3$ and $q \geq 5$. Suppose there exists an $(M)$-packing $\mathcal{P}$ of $G$ whose reduced leave is a $(p,q)$-chain $(x_1,x_2,\ldots,x_{p-1},y_0) \cdot (y_0,y_1,\ldots,y_{q-1})$ such that $y_0$ and $y_{q-2}$ are twin in $G$. Then there exists a repacking of $\mathcal{P}$ whose reduced leave is either
\begin{itemize}
	\item
a $(p+2,q-2)$-chain containing the $(q-2)$-cycle $(y_0,y_1,\ldots,y_{q-3})$; or
	\item
the $(p,q)$-chain $(x_1,x_2,\ldots,x_{p-1},y_0)\cdot(y_0,y_{q-1},y_{q-2},y_1,y_2,\ldots,y_{q-3})$.
\end{itemize}
\end{lemma}

\begin{proof}
Perform the $(y_0,y_{q-2})$-switch with origin $y_{q-3}$ (note that $y_0$ and $y_{q-2}$ are twin in $G$ and that $q \geq 5$). If the terminus of the switch is $y_1$, then the reduced leave of the resulting packing is the $(p,q)$-chain $(x_1,x_2,\ldots,x_{p-1},y_0)\cdot(y_0,y_{q-1},y_{q-2},y_1,y_2,\ldots,y_{q-3})$. Otherwise the terminus of the switch is $x_1$ or $x_{p-1}$ and in either case the leave of the resulting packing is a $(p+2,q-2)$-chain containing the $(q-2)$-cycle $(y_0,y_1,\ldots,y_{q-3})$.
\end{proof}

\begin{lemma}\label{2PureEdgesOddCycles}
Let $U$ and $V$ be sets with $U\subseteq V$ and $|U|,|V|$ odd, and let $M$ be a list of integers. Let $m$, $p$ and $q$ be positive integers with $m$ odd and $m, p+q-m\geq 3$.  Suppose there exists an $(M)$-packing $\mathcal{P}$ of $K_V-K_U$ whose reduced leave $L$ is a $(p,q)$-chain such that each cycle of $L$ contains exactly one pure edge and the link vertex of $L$ is in $V\setminus U$ if $3\in\{m,p+q-m\}$. Then there exists a repacking of $\mathcal{P}$ whose reduced leave is the edge-disjoint union of an $m$-cycle and a $(p+q-m)$-cycle.
\end{lemma}

\begin{proof}
We can assume without loss of generality that $m\geq p+q-m$ and that $p \leq q$. Note that this implies $p\leq m$. Since each cycle of $L$ must contain an even number of cross edges, $p$ and $q$ are odd. If $p=m$, then we are finished immediately, so we can assume that $p \leq m-2$. We will show that we can obtain a repacking of $\mathcal{P}$ whose reduced leave is either a $(p+2,q-2)$-chain or the edge-disjoint union of an $m$-cycle and a $(p+q-m)$-cycle. This will suffice to complete the proof, because by iteratively applying this procedure we will eventually obtain a repacking of $\mathcal{P}$ whose reduced leave is the edge-disjoint union of an $m$-cycle and a $(p+q-m)$-cycle.

\noindent{\bf Case 1.} Suppose that $L$ can be labelled as $(x_1,x_2,\ldots,x_{p-1},y_0)\cdot(y_0,y_1,\ldots, y_{q-1})$ so that $y_0x_1$ is not a pure edge and $y_ry_{r+1}$ is a pure edge (subscripts modulo $q$) for an integer $r$ such that $m-p+2 \leq r \leq q-1$. Then the hypotheses of Lemma~\ref{FigureOfEight3}(i) are satisfied and we can apply it to obtain a repacking of $\mathcal{P}$ whose reduced leave is
the edge-disjoint union of an $m$-cycle and a $(p+q-m)$-cycle.

\noindent{\bf Case 2.} Suppose that $L$ cannot be labelled as in Case 1. Without loss of generality we can label $L$ as $(x_1,x_2,\ldots,x_{p-1},y_0)\cdot(y_0,y_1,\ldots, y_{q-1})$ so that $y_0x_1$ is not a pure edge and $y_ry_{r+1}$ is a pure edge (subscripts modulo $q$) for an integer $r$ such that $\frac{q-1}{2} \leq r \leq q-1$, $r$ is even if $y_0 \in V \setminus U$, and $r$ is odd if $y_0 \in U$. It must be that $r\leq m-p+1$, for otherwise we would be in Case 1.
Then we can iteratively apply Lemma~\ref{Rearrange_Local} to obtain a repacking of $\mathcal{P}$ whose reduced leave $L'$ is either a $(p+2,q-2)$-chain or a $(p,q)$-chain which can be labelled  $(x'_1,x'_2,\ldots,x'_{p-1},y'_0) \cdot (y'_0,y'_1,\ldots, y'_{q-1})$ so that $y'_0x'_1$ is not a pure edge, and $y'_{r'}y'_{r'+1}$ is a pure edge (subscripts modulo $q$), where $r'$ is the element of $\{m-p+2,m-p+3\}$ such that $r' \equiv r \mod{2}$. Note that $r' \leq q-1$ because if $p+q-m \geq 4$ then $m-p+3 \leq q-1$, and if $p+q-m=3$ then $y_0 \in V \setminus U$, $r$ is even and $r'=m-p+2=q-1$. If $L'$ is a $(p+2,q-2)$-chain then we are finished, and if $L'$ is a $(p,q)$-chain then we can proceed as we did in Case 1.
\end{proof}

\begin{lemma}\label{2PureEdgesEvenCycles}
Let $U$ and $V$ be sets with $U\subseteq V$ and $|U|,|V|$ odd, and let $M$ be a list of integers.  Let $m$, $p$ and $q$ be positive integers with $m$ odd and $m, p+q-m\geq 3$. Suppose there exists an $(M)$-packing $\mathcal{P}$ of $K_V-K_U$ whose reduced leave $L$ is a $(p,q)$-chain such that one cycle in $L$ contains no pure edges, the other contains exactly two pure edges, and the link vertex of $L$ is in $V\setminus U$ if $3\in\{m,p+q-m\}$. Then there exists a repacking of $\mathcal{P}$ whose reduced leave is the edge-disjoint union of an $m$-cycle and a $(p+q-m)$-cycle.
\end{lemma}

\begin{proof}
We can assume without loss of generality that $m\geq p+q-m$ and that a $p$-cycle in $L$ contains no pure edges. Since each cycle of $L$ must contain an even number of cross edges, $p$ and $q$ are even.

\noindent{\bf Case 1.} Suppose that $L$ can be labelled as $(x_1,x_2,\ldots,x_{p-1},y_0)\cdot(y_0,y_1,\ldots, y_{q-1})$ so that $y_ry_{r+1}$ and $y_sy_{s+1}$ are pure edges (subscripts modulo $q$) for integers $r$ and $s$ such that $0 \leq r < s \leq q-1$, $r \leq m-2$ and $s \geq m-p+1$. Observe that, in particular, such a labelling is always possible when $q=4$ (any labelling with $r < s$ and $s\in\{2,3\}$ will suffice, because then $r\leq 2< m-2$ since $m\geq p+4-m$ and $m-p+1\leq 2\leq s$ since $p+4-m\geq 3$). Let $x_0=y_0$ and
$t=\max(r+1,m-p+1)$. Consider the vertices $x_{m-t}$ and $y_t$. Note that $1 \leq m-t \leq p-1$ because $r \leq m-2$, $p \geq 3$ and $t \geq m-p+1$, and that $r < t \leq s$ because $t \geq r+1$, $r<s$ and $s \geq m-p+1$. Since $r +1\leq t \leq s$, there is exactly one pure edge in the $m$-path $[x_{m-t},x_{m-t-1},\ldots,x_1,y_0,y_1,\ldots,y_t]$ and hence $x_{m-t}$ and $y_t$ are twin in $K_V-K_U$. Let $L'$ be the reduced leave of the repacking of $\mathcal{P}$ obtained by performing the $(x_{m-t},y_t)$-switch with origin $x_{m-t-1}$.
If the terminus of the switch is not $y_{t-1}$, $L'$ is the edge-disjoint union of an $m$-cycle and a $(p+q-m)$-cycle and we are finished. If the terminus of the switch is $y_{t-1}$, then $L'$ is a $(p+2t-m,q+m-2t)$-chain with one pure edge in each cycle and whose link vertex is in $V\setminus U$ if $3\in\{m,p+q-m\}$, and we can apply Lemma \ref{2PureEdgesOddCycles} to complete the proof.

\noindent{\bf Case 2.} Suppose that $L$ cannot be labelled as in Case 1. From our comments in Case 1 we may assume $q \geq 6$.  We will show that we can obtain a repacking of $\mathcal{P}$ whose reduced leave either satisfies the conditions of Case 1 or is a $(p+2,q-2)$-chain in which a $(p+2)$-cycle contains no pure edges. Since any reduced leave which is a $(p+q-4,4)$-chain with exactly two pure edges in which a $(p+q-4)$-cycle contains no pure edges must fall into Case 1, repeating this procedure will eventually result in a repacking of $\mathcal{P}$ whose reduced leave satisfies the conditions of Case 1. We can then proceed as we did in Case 1 to complete the proof.

Without loss of generality we can label $L$ as $(x_1,x_2,\ldots,x_{p-1},y_0)\cdot(y_0,y_1,\ldots, y_{q-1})$ so that $y_ry_{r+1}$ and $y_sy_{s+1}$ are pure edges (subscripts modulo $q$) for integers such that $0 \leq r < s \leq q-1$ and $r \leq \frac{q}{2}$. Because $r \leq \frac{q}{2}$ and $\frac{q}{2} \leq m-2$ (note that $m \geq \frac{p+q}{2} \geq \frac{q+4}{2}$), it must be that $s < m-p+1$, for otherwise we would be in Case 1. So we can repeatedly apply Lemma~\ref{Rearrange_Local} to obtain a repacking of $\mathcal{P}$ whose reduced leave $L'$ is either a $(p+2,q-2)$-chain in which a $(p+2)$-cycle contains no pure edges or a $(p,q)$-chain which can be labelled  $(x'_1,x'_2,\ldots,x'_{p-1},y'_0) \cdot (y'_0,y'_1,\ldots, y'_{q-1})$ so that $y'_ry'_{r+1}$ and $y'_{s'}y'_{s'+1}$ are pure edges for integers $r'$ and $s'$ such that $0 \leq r' < s' \leq q-1$ and $s' \in \{m-p+1,m-p+2\}$ (note that $m-p+2 \leq q-1$ since $p+q-m \geq 3$). Observe that in the latter case $L'$ satisfies the conditions of Case 1.
\end{proof}

\begin{lemma}\label{2PureEdges}
Let $U$ and $V$ be sets with $U\subseteq V$ and $|U|,|V|$ odd, and let $M$ be a list of integers.  Let $m$, $p$ and $q$ be positive integers with $m$ odd, and $m, p+q-m\geq 3$. Suppose there exists an $(M)$-packing of $K_V-K_U$ whose reduced leave $L$ is a $(p,q)$-chain such that $L$ contains exactly two pure edges and the link vertex of $L$ is in $V\setminus U$ if $3\in\{m,p+q-m\}$. Then there exists a repacking of $\mathcal{P}$ whose reduced leave is the edge-disjoint union of an $m$-cycle and a $(p+q-m)$-cycle.
\end{lemma}

\begin{proof}
If each cycle of $L$ contains exactly one pure edge, then we can apply Lemma \ref{2PureEdgesOddCycles} to complete the proof. If one cycle in $L$ contains no pure edges and the other contains exactly two pure edges, then we can apply Lemma \ref{2PureEdgesEvenCycles} to complete the proof.
\end{proof}

\section{Packings whose leaves are $t$-chains}\label{tChainSec}

In this section we use Lemma \ref{2PureEdges} to prove an analogous result for chains with more than two cycles, namely Lemma \ref{2Cycles}. Given a packing whose reduced leave is an $s$-chain that contains two pure edges and satisfies certain other properties, Lemma \ref{2Cycles} allows us to find a repacking whose reduced leave is the edge-disjoint union of two cycles of specified lengths. Lemmas \ref{ReducePathLength} and \ref{2Paths} are used only in order to prove Lemma \ref{2Cycles}, while Lemma \ref{Deg4AndIsolate} will also be used in Section \ref{JoiningSec}. We will need the following additional definitions for chains and rings.

\begin{definition} For sets $U$ and $V$ with $U \subseteq V$, an $s$-chain that is a subgraph of $K_V-K_U$ is \emph{good} if $s=2$ or if $s \geq 3$ and
\begin{itemize}
\item one end cycle of the chain contains at least one pure edge and has its link vertex in $V \setminus U$; and
\item each internal cycle of the chain has one link vertex in $V \setminus U$ and one link vertex in $U$.
\end{itemize}
\end{definition}

\begin{definition} For sets $U$ and $V$ with $U \subseteq V$, an $s$-ring that is a subgraph of $K_V-K_U$ is \emph{good} if either
\begin{itemize}
	\item
$s$ is even, and each of the ring cycles has one link vertex in $U$ and one link vertex in $V \setminus U$; or
	\item
$s$ is odd, one ring cycle has both link vertices in $V \setminus U$ and contains at least one pure edge, and each other ring cycle has one link vertex in $U$ and one link vertex in $V \setminus U$.
\end{itemize}
\end{definition}

\begin{lemma}\label{Deg4AndIsolate}
Let $U$ and $V$ be sets with $U\subseteq V$ and $|U|,|V|$ odd, and suppose that $L$ is a subgraph of $K_V-K_U$ such that $L$ contains exactly two pure edges and each vertex of $L$ has positive even degree.
\begin{itemize}
    \item[(i)]
If $|E(L)| \leq 2(|U|+1)$ and $U$ contains a vertex of degree at least $4$ in $L$, then there is a vertex $y$ in $U$ such that $y \notin V(L)$.
    \item[(ii)]
If $|E(L)| \leq 2\min(|U|+2,|V|-|U|)$ and $S$ is an element of $\{U,V \setminus U\}$ such that $S$ contains either at least two vertices of degree $4$ in $L$ or at least one vertex of degree at least $6$ in $L$, then there is a vertex $y$ in $S$ such that $y \notin V(L)$.
    \item[(iii)]
If $|E(L)| \leq 2\min(|U|+2,|V|-|U|)$ and $L$ contains either at least two vertices of degree $4$ or at least one vertex of degree at least $6$, then there are twin vertices $x$ and $y$ in $K_V-K_U$ such that $\deg_L(x) \geq 4$ and $y \notin V(L)$.
\end{itemize}
\end{lemma}

\begin{proof}
Let $l= |E(L)|$. Because $L$ contains exactly two pure edges, we have
$$\medop\sum_{x \in V(L) \cap U}\deg_L(x) = l-2 \mbox{\qquad and \qquad}  \medop\sum_{x \in V(L) \setminus U}\deg_L(x) = l+2.$$

\noindent{\bf Proof of (i).} Suppose that $l \leq 2(|U|+1)$ and $U$ contains a vertex of degree at least $4$ in $L$. Suppose for a contradiction that $U \subseteq V(L)$. Then we have $l-2 = \sum_{x \in V(L) \cap U}\deg_L(x) \geq 2|U|+2$ since every vertex of $L$ in $U$ has degree at least 2. This contradicts $l \leq 2(|U|+1)$.

\noindent{\bf Proof of (ii).} Suppose that $l \leq 2\min(|U|+2,|V|-|U|)$ and $S$ is an element of $\{U,V \setminus U\}$ such that $S$ contains either at least two vertices of degree $4$ in $L$ or at least one vertex of degree at least $6$ in $L$. Suppose for a contradiction that $S \subseteq V(L)$. Then we have $\medop\sum_{x \in V(L) \cap S}\deg_L(x) \geq 2|S|+4$ since every vertex of $L$ in $S$ has degree at least 2. So, if $S=U$, then $l-2 \geq 2|U|+4$, contradicting $l \leq 2(|U|+2)$. If $S=V \setminus U$, then $l+2 \geq 2(|V|-|U|)+4$, contradicting $l \leq 2(|V|-|U|)$.

\noindent{\bf Proof of (iii).} Because we have proved (ii), it only remains to show that if $L$ contains two vertices of degree $4$, one in $U$ and one in $V \setminus U$, and every other vertex of $L$ has degree $2$, then there are twin vertices $x$ and $y$ in $K_V-K_U$ such that $\deg_L(x) \geq 4$ and $y \notin V(L)$. Suppose otherwise. Then it must be the case that $V(L)=V$, $l-2=2|U|+2$ and $l+2=2(|V|-|U|)+2$. But then $l=2|U|+4$ and $l=2(|V|-|U|)$, so $|V|=2|U|+2$ which contradicts the fact that $|V|$ is odd.
\end{proof}

\begin{lemma}\label{ReducePathLength}
Let $U$ and $V$ be sets with $U\subseteq V$ and $|U|,|V|$ odd, and let $M$ be a list of integers. Let $p$ and $s$ be positive integers such that $p\geq 5$ is odd and $s\geq 2$. Suppose there exists an $(M)$-packing $\mathcal{P}$ of $K_V-K_U$ whose reduced leave $L$ is a good $s$-chain that has a decomposition $\{P,L-P\}$ into two paths such that $P$ has length $p$ and each path contains exactly one pure edge and has both end vertices in $V \setminus U$. Suppose further that $P$ has a subpath $P_0=[x_0,\ldots,x_r]$ such that $2 \leq r \leq p-1$, $x_0$ is an end vertex of $P$, $P_0$ contains no pure edge, and $\deg_L(x_{r-1})=\deg_L(x_r)=2$. Then there is a repacking of $\mathcal{P}$ whose reduced leave $L'$ is a good $s$-chain that has a decomposition $\{P',L'-P'\}$ into two paths such that $P'$ has length $p-2$, each path contains exactly one pure edge and has both end vertices in $V \setminus U$, and $P'$ contains a pure edge in an end cycle of $L'$ with link vertex in $V \setminus U$ if $P$ contains a pure edge in an end cycle of $L$ with link vertex in $V \setminus U$.
\end{lemma}

\begin{proof}
We prove the result by induction on the length of $P_0$. If $|E(P_0)|=2$, then $\{P',L-P'\}$ where $P'=[x_2,\ldots,x_p]$ is a decomposition of $L$ with the required properties. So we can assume that $|E(P_0)| \geq 3$. By induction we can assume that $P_0$ is the shortest subpath of $P$ satisfying the required conditions. Because $r\geq 3$ this implies $\deg_L(x_{r-2})=4$. Label the vertices in $V(P) \setminus V(P_0)$ so that $P=[x_0,\ldots,x_p]$.

The vertices $x_r$ and $x_{r-2}$ are twin in $K_V-K_U$ because they are joined by a path of length 2 containing no pure edge. Let $L'$ be the reduced leave of the repacking of $\mathcal{P}$ obtained by performing the $(x_r,x_{r-2})$-switch with origin $x_{r-3}$. Note that $L'$ is a good $s$-chain irrespective of the terminus of the switch. If the terminus of the switch is not $x_{r+1}$, then $\{P',L'-P'\}$ where $P'=[x_0,x_1,\ldots,x_{r-3},x_r,x_{r+1},\ldots,x_p]$ is a decomposition of $L'$ with the required properties. If the terminus of the switch is $x_{r+1}$, then $\{P',L'-P'\}$ where $P'=[x_0,x_1,\ldots,x_{r-3},x_r,x_{r-1},x_{r-2},x_{r+1},x_{r+2},\ldots,x_p]$ is a decomposition of $L'$ into two paths such that $P'$ has length $p$ and each path contains exactly one pure edge and has both end vertices in $V \setminus U$. Further $P'$ has the subpath $P'_0=[x_0,\ldots,x_{r-3},x_{r},x_{r-1}]$ and we know that $x_0$ is an end vertex of $P'$, $P'_0$ contains no pure edge, and $\deg_{L'}(x_{r})=\deg_{L'}(x_{r-1})=2$. Thus, because $|E(P'_0)|=r-1$, we are finished by our inductive hypothesis.
\end{proof}

\begin{lemma}\label{2Paths}
Let $U$ and $V$ be sets with $U\subseteq V$ and $|U|,|V|$ odd, and let $M$ be a list of integers. Let $m_1$, $m_2$ and $s$ be positive integers such that $m_1$ and $m_2$ are odd, $m_1, m_2 \geq s$ and $s\geq 3$. Suppose there exists an $(M)$-packing $\mathcal{P}$ of $K_V-K_U$ whose reduced leave is a good $s$-chain of size $m_1+m_2$ that contains exactly two pure edges. Then there exists a repacking of $\mathcal{P}$ whose reduced leave is a good $s$-chain that has a decomposition into an $m_1$-path and an $m_2$-path such that each path contains exactly one pure edge.
\end{lemma}

\begin{proof} Suppose without loss of generality that $m_1 \leq m_2$, and let $L$ be the reduced leave of $\mathcal{P}$. Note that $|E(L)|=m_1+m_2$. Because $L$ is good and contains exactly two pure edges, we can find some decomposition $\{P,L-P\}$ of $L$ into two odd length paths each of which has both end vertices in $V \setminus U$ and contains exactly one pure edge. Without loss of generality we can assume that $P$ is at least as long as $L-P$ if $m_1 \geq s+1$ and that $P$ contains a pure edge in an end cycle of $L$ with link vertex in $V \setminus U$ if $m_1 = s$. Let $p$ be the length of $P$ and note that in each case $p\geq m_1$ because $p \geq \frac{m_1+m_2}{2} \geq m_1$ if $m_1 \geq s+1$ and $p \geq s = m_1$ if $m_1=s$.  We are finished if $p=m_1$, so we may assume $p\geq m_1+2$.

\noindent{\bf Case 1.} Suppose each cycle of $L$ contains at most two edges of $P$. Then exactly $p-s$ cycles of $L$ contain two edges of $P$ and the rest contain one edge of $P$. Because $L$ is good and both end vertices of $P$ are in $V \setminus U$, if $C$ is a cycle of $L$ that contains two edges of $P$, then either
\begin{itemize}
    \item
$C$ is an internal cycle of $L$ and $C$ contains the pure edge of $P$; or
    \item
$C$ is an end cycle of $L$ with link vertex in $U$ and $C$ contains the pure edge of $P$; or
    \item
$C$ is an end cycle of $L$ with link vertex in $V \setminus U$ and $C$ does not contain the pure edge of $P$.
\end{itemize}
From this it follows that $p-s \leq 3$. Note that $p \geq m_1+2 \geq s+2$ and hence that $p \in \{s+2,s+3\}$. If $p=s+2$, then $m_1=s$. But then $P$ contains a pure edge in an end cycle of $L$ with link vertex in $V \setminus U$ by its definition and it can be seen that no cycle of $L$ contains two edges of $P$, contradicting $p=s+2$. So it must be that $p=s+3$ and thus $m_1=s+1=p-2$ because $m_1$ and $p$ are odd. Because $p=s+3$, $P$ contains two edges of each end cycle of $L$ and two edges, including a pure edge, of some internal cycle of $L$. Let $P'$ be the path obtained from $P$ by deleting both end vertices of $P$ and their incident edges. Then $\{P',L-P'\}$ is a decomposition of $L$ into an $m_1$-path and an $m_2$-path such that each path contains exactly one pure edge.

\noindent{\bf Case 2.} Suppose there is a cycle $C$ in $L$ such that $C \cap P$ is a path of length at least 3. Let $P_0=[x_0,\ldots,x_r]$ be a subpath of $P$ such that $x_0$ is an end vertex of $P$, $P_0$ contains no pure edge, and $P_0$ contains exactly two edges in $C \cap P$. If $C \cap P$ contains no pure edge or if $C \cap P$ has length at least 4, then it is easy to see such a subpath exists. If $C \cap P$ has length 3 and contains a pure edge, then the facts that $L$ is good and that the end vertices of $P$ are in $V \setminus U$ imply that $C$ is an end cycle of $L$ with link vertex in $V \setminus U$ and hence that such a subpath exists. So we can apply Lemma~\ref{ReducePathLength} to obtain a repacking of $\mathcal{P}$ whose reduced leave $L'$ is a good $s$-chain that has a decomposition $\{P',L'-P'\}$ into two paths such that $P'$ has length $p-2$, each path contains exactly one pure edge and has both end vertices in $V \setminus U$, and $P'$ contains a pure edge in an end cycle of $L'$ with link vertex in $V \setminus U$ if $m_1 = s$. It is clear that by repeating this procedure we will eventually obtain a repacking of $\mathcal{P}$ whose reduced leave either has a decomposition into an $m_1$-path and an $m_2$-path such that each path contains exactly one pure edge or has a decomposition into odd length paths which satisfies the hypotheses for Case 1. In the former case we are finished and in the latter we can proceed as we did in Case 1.
\end{proof}

\begin{lemma}\label{2Cycles}
Let $U$ and $V$ be sets with $U\subseteq V$ and $|U|,|V|$ odd, and let $M$ be a list of integers. Let $m_1$, $m_2$ and $s$ be positive integers such that $s\geq 2$, $m_1$ and $m_2$ are odd, $m_1,m_2\geq s$, $m_1+m_2\leq 2\min(|U|+2, |V|-|U|)$, and $m_1+m_2\leq 2(|U|+1)$ if $3\in\{m_1,m_2\}$.  Suppose there exists an $(M)$-packing $\mathcal{P}$ of $K_V-K_U$ whose reduced leave has size $m_1+m_2$, contains exactly two pure edges, is either a good $s$-ring or a good $s$-chain that, if $3\in\{m_1,m_2\}$, is not a $2$-chain with link vertex in $U$. Then there exists a repacking of $\mathcal{P}$ whose reduced leave is the edge-disjoint union of an $m_1$-cycle and an $m_2$-cycle.
\end{lemma}

\begin{proof}
Let $L$ be the reduced leave of $\mathcal{P}$. We first show that the result holds for $s=2$. If $L$ is a $2$-chain, then the result follows by Lemma~\ref{2PureEdges}. If $L$ is a $2$-ring, then it follows from our hypotheses and Lemma \ref{Deg4AndIsolate} that there are twin vertices $x$ and $y$ in $K_V-K_U$ such that $\deg_L(x) \geq 4$ and $y \notin V(L)$, and such that if $3\in\{m_1,m_2\}$ then $x \in U$ (if $3\in\{m_1,m_2\}$, then apply Lemma \ref{Deg4AndIsolate}(i) and otherwise apply Lemma \ref{Deg4AndIsolate}(iii)). Performing an $(x,y)$-switch results in a repacking of $\mathcal{P}$ whose reduced leave is a $2$-chain whose link vertex is in $V \setminus U$ if $3\in\{m_1,m_2\}$ and the result follows by Lemma~\ref{2PureEdges}. So it is sufficient to show, for each integer $s' \geq 3$, that if the result holds for $s=s'-1$ then it holds for $s=s'$.

\noindent{\bf Case 1.} Suppose that $L$ is a good $s'$-chain. By Lemma~\ref{2Paths} we can obtain a repacking of $\mathcal{P}$ whose reduced leave is a good $s'$-chain with a decomposition into paths of length $m_1$ and $m_2$ each containing exactly one pure edge. Let $[x_0,x_1,\ldots,x_{m_1}]$ be the path of length $m_1$. Observe that $x_0$ and $x_{m_1}$ are twin in $K_V-K_U$ because they are joined by an odd length path containing exactly one pure edge, and perform the $(x_0,x_{m_1})$-switch with origin $x_1$.

If the terminus of the switch is not $x_{m_1-1}$, then we obtain a repacking of $\mathcal{P}$ whose reduced leave is the edge-disjoint union of an $m_1$-cycle and an $m_2$-cycle and we are finished. If the terminus of the switch is $x_{m_1-1}$, then we obtain a repacking of $\mathcal{P}$ whose reduced leave is a good $(s'-1)$-ring that contains exactly two pure edges and the result follows by our inductive hypothesis.

\noindent{\bf Case 2.} Suppose that $L$ is a good $s'$-ring. Let $A$ be a ring cycle of $L$ such that $A$ contains a pure edge and if $s'$ is odd then $A$ has both link vertices in $V \setminus U$. Let $x$ and $y$ be twin vertices in $K_V-K_U$ such that $x$ is a link vertex in $A$, $x\in U$ if $s'$ is even, and $y \notin V(L)$. Such a vertex $y$ exists by Lemma \ref{Deg4AndIsolate}(ii) because $|E(L)| \leq 2\min(|U|+2, |V|-|U|)$, $V \setminus U$ contains two vertices of degree 4 in $L$ if $s'$ is odd, and $U$ contains two vertices of degree 4 in $L$ if $s'$ is even (for then $s' \geq 4$). By performing an $(x,y)$-switch with origin in $V(A)$ we obtain a repacking of $\mathcal{P}$ whose reduced leave contains exactly two pure edges, is a good $s'$-chain if the terminus of the switch is also in $V(A)$, and is a good $(s'-1)$-ring otherwise. In the former case we can proceed as in Case 1 and in the latter case the result follows by our inductive hypothesis.
\end{proof}

\section{Merging cycle lengths}\label{JoiningSec}

In this section we use Lemma \ref{2Cycles} to prove Lemma \ref{JoiningLemma}, which is the key result in proving Theorem~\ref{MainTheorem}. Given a cycle decomposition of $K_v-K_u$ that satisfies certain conditions, Lemma \ref{JoiningLemma} allows us to find a new cycle decomposition of $K_v-K_u$ in which some of the shorter cycle lengths have been merged into cycles of length $m$.

\begin{lemma}\label{1Degree4Vertex}
Let $U$ and $V$ be sets with $U\subseteq V$ and $|U|,|V|$ odd, and let $M$ be a list of integers. Let $m_1$, $m_2$, $t$ and $k$ be positive integers such that $m_1$ and $m_2$ are odd, $m_1,m_2 \geq k+t-1$, $m_1+m_2\leq 2\min(|U|+2,|V|-|U|)$, and  $m_1+m_2\leq 2(|U|+1)$ if $3\in\{m_1,m_2\}$. Suppose there exists an $(M)$-packing $\mathcal{P}$ of $K_V-K_U$ with a reduced leave $L$ of size $m_1+m_2$ such that $L$ contains exactly two pure edges and $L$ has exactly $k$ components, $k-1$ of which are cycles and one of which is a good $t$-chain that, if $3\in\{m_1,m_2\}$, is not a $2$-chain with link vertex in $U$. Then there exists a repacking of $\mathcal{P}$ whose reduced leave is the edge-disjoint union of an $m_1$-cycle and an $m_2$-cycle.
\end{lemma}

\begin{proof}
By Lemma~\ref{2Cycles} it is sufficient to show that we can construct a repacking of $\mathcal{P}$ whose reduced leave is a good $s$-chain, for some $s \in \{2,\ldots,k+t-1\}$, that is not a $2$-chain with link vertex in $U$ if $3\in\{m_1,m_2\}$. If $k=1$, then we are finished, so we can assume $k \geq 2$. By induction on $k$, it suffices to show that there is a repacking of $\mathcal{P}$ with a reduced leave $L'$ such that $L'$ has exactly $k-1$ components, one component of $L'$ is a good $t'$-chain for $t' \in \{t,t+1\}$, each other component of $L'$ is a cycle, and a degree $4$ vertex of $L'$ is in $V \setminus U$ if $3\in\{m_1,m_2\}$.

Let $H$ be the component of $L$ which is a good $t$-chain, and let $C$ be a component of $L$ such that $C$ is a cycle and $C$ contains at least one pure edge if $H$ contains at most one pure edge. Let $H_1$ and $H_t$ be the end cycles of $H$ where $H_1$ contains a pure edge if $H$ does and the link vertex of $H_1$ is in $V \setminus U$ if $t \geq 3$.

\noindent{\bf Case 1.} Suppose that either $t \geq 3$ or it is the case that $t=2$, $H_1$ contains a pure edge, and the link vertex of $H$ is in $V \setminus U$. Let $x$ and $y$ be vertices such that $x\in V(H_t)$, $x$ is not a link vertex of $H$, $y\in V(C)$, $x,y\in V \setminus U$ if $t$ is odd, and $x,y\in U$ if $t$ is even. Let $\mathcal{P}'$ be a repacking of $\mathcal{P}$ obtained by performing an $(x,y)$-switch with origin in $V(H_t)$. The reduced leave $L'$ of $\mathcal{P}'$ has exactly $k-1$ components, $k-2$ of which are cycles and one of which is a good $t'$-chain, where $t'=t+1$ if the terminus of the switch is also in $V(H_t)$ and $t'=t$ otherwise. Further, a degree $4$ vertex of $L'$ is in $V \setminus U$ if $3\in\{m_1,m_2\}$.  So we are finished by induction.

\noindent{\bf Case 2.} Suppose that $t=2$ and either $H$ contains exactly one pure edge and has its link vertex in $U$ or $H$ contains no pure edges. Then $C$ contains a pure edge. Let $w$ and $x$ be vertices such that $w\in V(C) \setminus U$, $x\in V(H_1) \setminus U$, and $x$ is not the link vertex of $H$. Let $\mathcal{P}'$ be a repacking of $\mathcal{P}$ obtained by performing a $(w,x)$-switch with origin in $V(H_1)$ and let $L'$ be the reduced leave of $\mathcal{P}'$. If the terminus of this switch is in $C$, then $L'$ has exactly $k-1$ components, $k-2$ of which are cycles and one of which is a $2$-chain, and the link vertex of this chain is in $V \setminus U$ if $3\in\{m_1,m_2\}$. In this case we are finished by induction. Otherwise the terminus of this switch is in $V(H_1)$ and $L'$ has exactly $k-1$ components, $k-2$ of which are cycles and one of which is a $3$-chain $H'$ one of whose end cycles contains a pure edge and has its link vertex in $V \setminus U$. If $H'$ is good, then we are done. Otherwise, it must be that both link vertices of $H'$ are in $V \setminus U$. In this latter case we proceed as follows.

Let $H'_1$ and $H'_3$ be the end cycles of $H'$ such that $H'_1$ has a pure edge.
Let $y,z\in V \setminus U$ be vertices such that $y$ is the link vertex in $V(H'_3)$ and $z \notin V(L')$ (note that $z$ exists by Lemma \ref{Deg4AndIsolate}(ii) because $m_1+m_2\leq 2\min(|U|+2,|V|-|U|)$ and both link vertices of $H'$ are in $V \setminus U$). Let $\mathcal{P}''$ be a repacking of $\mathcal{P}$ obtained from $\mathcal{P}'$ by performing a $(y,z)$-switch with origin in $V(H'_3)$ and let $L''$ be the reduced leave of $\mathcal{P}''$. If the terminus of this switch is not in $V(H'_3)$, then $L''$ has exactly $k-1$ components, $k-2$ of which are cycles and one of which is a $2$-chain whose link vertex is in $V \setminus U$. In this case we are finished by induction. Otherwise, the terminus of this switch is in $V(H'_3)$ and $L''$ has exactly $k$ components, $k-1$ of which are cycles and one of which is a $2$-chain that contains a pure edge and has its link vertex in $V \setminus U$. In this case we can proceed as we did in Case 1.

\noindent{\bf Case 3.} Suppose that $t=2$, $H$ contains two pure edges and the link vertex of $H$ is in $U$. Note that, from our hypotheses, $m_1,m_2 \geq 4$. Let $x$ be the link vertex of $H$ and let $y$ be a vertex in $V(C) \cap U$. Let $\mathcal{P}'$ be a repacking of $\mathcal{P}$ obtained by performing an $(x,y)$-switch with origin in $V(H_2)$ and let $L'$ be the reduced leave of $\mathcal{P}'$. If the terminus of this switch is in $V(C)$, then $L'$ has exactly $k-1$ components, $k-2$ of which are cycles and one of which is a $2$-chain. In this case we are finished by induction. Otherwise the terminus of this switch is in $V(H_2)$ and $L'$ has exactly $k$ components, $k-1$ of which are cycles and one of which is a $2$-chain that contains at most one pure edge and has its link vertex in $U$. In this case we can proceed as we did in Case 2.
\end{proof}

\begin{lemma}\label{MaxComponents}
Let $U$ and $V$ be sets with $U\subseteq V$ and $|U|,|V|$ odd. If $L$ is a subgraph of $K_V-K_U$ such that $L$ contains at most two pure edges, $L$ has one vertex of degree $4$, and each other vertex of $L$ has degree $2$, then $L$ has at most $\left\lfloor\frac{|E(L)|-6}{4}\right\rfloor+1$ components.
\end{lemma}

\begin{proof}
Because each vertex of $L$ has even degree, $L$ has a decomposition $\mathcal{D}$ into cycles. Since there are at most two pure edges in $L$, at most two cycles in $\mathcal{D}$ have length 3 and each other cycle in $\mathcal{D}$ has length at least $4$. Thus $|E(L)| \geq 4(|\mathcal{D}|-2)+6$ which implies $|\mathcal{D}| \leq \left\lfloor\frac{|E(L)|-6}{4}\right\rfloor+2$. At least one component of $L$ contains a vertex of degree 4 and hence contains at least two cycles and each other component of $L$ contains at least one cycle. The result follows.
\end{proof}

\begin{lemma}\label{PickApart}
Let $U$ and $V$ be sets with $U\subseteq V$ and $|U|,|V|$ odd, and let $M$ be a list of integers.
Suppose there exists an $(M)$-packing $\P_0$ of $K_V-K_U$ with a reduced leave $L_0$ such that $|E(L_0)|\leq 2\min(|U|+2,|V|-|U|)$, $L_0$ has exactly two pure edges, and $L_0$ has at least one vertex of degree at least $4$. Then there exists a repacking $\P^{\star}$ of $\P_0$ with a reduced leave $L^{\star}$ such that exactly one vertex of $L^{\star}$ has degree $4$ and every other vertex of $L^{\star}$ has degree $2$.
\end{lemma}

\begin{proof}
Let $d=\frac{1}{2}\sum_{x\in V(L_0)} (\deg_{L_0}(x)-2)$, and construct a sequence $\P_0, \P_1,\ldots, \P_{d-1}$, where for $i \in \{0,\ldots,d-2\}$ $\P_{i+1}$ is a repacking of $\P_i$ obtained from $\P_i$ by  performing an $(x_i,y_i)$-switch where $x_i$ and $y_i$ are twin vertices in $K_V-K_U$ such that the degree of $x_i$ in the reduced leave of $\P_i$ is at least $4$ and $y_i$ is not in the reduced leave of $\P_i$. Such vertices exist by Lemma \ref{Deg4AndIsolate}(iii) since $|E(L_0)|\leq 2\min(|U|+2,|V|-|U|)$ and $i\leq d-2$. Exactly one vertex of the reduced leave of $\P_{d-1}$ has degree $4$ and all its other vertices have degree $2$.
\end{proof}

\begin{lemma}\label{JoiningLemma}
Let $U$ and $V$ be sets with $U\subseteq V$ and $|U|,|V|$ odd, and let $M$ be a list of integers.
Let $m$ be a positive odd integer such that $7 \leq m \leq \min(|U|+2,|V|-|U|-1)$. Let $a_1,\ldots,a_s$ and $b_1,\ldots,b_t$ be lists of integers such that $a_1+\cdots+a_s=m$ and $b_1+\cdots+b_t=m$.
Suppose there exists an $(M)$-packing $\mathcal{P}$ of $K_V-K_U$ with a reduced leave that
contains exactly two pure edges and is the edge-disjoint union of cycles of lengths $a_1,\ldots,a_s,b_1,\ldots,b_t$. Then there exists an $(M,m,m)$-decomposition $\mathcal{D}$ of $K_V-K_U$ containing two $m$-cycles $C'$ and $C''$ such that $\mathcal{D}\setminus \{C',C''\}$ is a repacking of $\mathcal{P}$.
\end{lemma}

\begin{proof}
Let $L$ be the reduced leave of $\P$. It obviously suffices to find a repacking of $\mathcal{P}$ whose reduced leave is the edge-disjoint union of two $m$-cycles.

We prove the result by induction on $s+t$. If $s=1$ and $t=1$, then the result is trivial. So suppose that $s+t \geq 3$. Assume without loss of generality that $s \geq t$ and note that $s \geq 2$.

\noindent{\bf Case 1.} Suppose that some vertex of $L$ has degree at least $4$. Then by Lemma \ref{PickApart}, there is a repacking $\mathcal{P}'$ of $\mathcal{P}$ with a reduced leave $L'$ such that exactly one vertex of $L'$ has degree $4$ and every other vertex of $L'$ has degree $2$. So one component of $L'$ is a $2$-chain, and any other component of $L'$ is a cycle. Furthermore $L'$ contains at most $\lfloor\frac{2m-6}{4}\rfloor+1$ components by Lemma \ref{MaxComponents} and obviously $m \geq \lfloor\frac{2m-6}{4}\rfloor+2$. Thus, applying Lemma \ref{1Degree4Vertex} with $m_1=m_2=m$ to $\mathcal{P}'$, there is a repacking of $\mathcal{P}$ whose reduced leave is the edge-disjoint union of two $m$-cycles.

\noindent{\bf Case 2.} Suppose that every vertex of $L$ has degree $2$. Then the components of $L$ are cycles of lengths $a_1,\ldots,a_s,b_1,\ldots,b_t$. Let $x$ and $y$ be vertices in $V\setminus U$ such that $x$ and $y$ are in two distinct cycles of $L$ which have lengths $a_1$ and $a_2$ respectively. Let $\mathcal{P}'$ be a repacking of $\mathcal{P}$ obtained by performing an $(x,y)$-switch and let $L'$ be the reduced leave of $\mathcal{P}'$. If the origin and terminus of this switch are in the same cycle, then one vertex of $L'$ has degree $4$ and every other vertex of $L'$ has degree $2$, and we can proceed as we did in Case 1. If the origin and terminus of this switch are in different cycles, then $L'$ is the edge-disjoint union of cycles of lengths $a_1+a_2,a_3,\ldots,a_s,b_1,\ldots,b_t$ (lengths $a_1+a_2,b_1,\ldots,b_t$ if $s=2$) and we can complete the proof by applying our inductive hypothesis.
\end{proof}

\section{Base decompositions}\label{BaseDecompSec}

Our goal in this section is to prove Lemmas \ref{BaseDecomposition} and \ref{BaseDecomposition_9} which provide the ``base'' decompositions of $K_v-K_u$ into short cycles and $m$-cycles to which we apply Lemma \ref{JoiningLemma} in order to prove Theorem \ref{MainTheorem}. Lemma \ref{BaseDecomposition} is used in the case where $m \geq 11$ and Lemma \ref{BaseDecomposition_9} is used when $m=9$. We first require several preliminary results. Lemma \ref{MixedCycles} is a method for decomposing certain graphs into $3$-cycles and $5$-cycles. Theorems \ref{AlspachEven}, \ref{ChouFuHuang} and \ref{Bipartite} are existing results on decomposing the complete graph and the complete bipartite graph into cycles.

We require some additional notation in the remainder of the paper. For a positive integer $v$, let $K^c_v$ denote a graph of order $v$ with no edges and, for a set $V$, let $K^c_V$ denote the graph with vertex set $V$ and no edges. For a non-negative integer $i$, let $x^i$ denote a list containing $i$ entries all equal to $x$. For technical reasons we ignore any $0$'s in a list $M$ when discussing $(M)$-decompositions.

\begin{lemma}\label{MixedCycles}
Let $a$ and $k$ be non-negative integers such that $k\geq 3$, $a\leq k$ and $a$ is even. Let $C$ be a cycle of length $k$, and let $N$ be a vertex set of size $k-a$ such that $V(C)\cap N=\emptyset$. Then there exists a $(3^a,5^{k-a})$-decomposition of $K_2^c\vee (C\cup K_N^c)$ such that each cycle in the decomposition contains exactly one edge of $C$.
\end{lemma}

\begin{proof}
Let $y$ and $z$ be the vertices in $K_2^c$, let $C=(c_1,c_2,\ldots,c_k)$, and let $N=\{x_1,x_2,\ldots,x_{k-a}\}$. Let
\begin{align*}
\mathcal{D}_1 &= \{(y,c_k,c_1),(z,c_1,c_2),(y,c_2,c_3),(z,c_3,c_4),\ldots,(y,c_{a-2},c_{a-1}),(z,c_{a-1},c_a)\}; \hbox{ and}\\
\mathcal{D}_2 &= \{(y,c_a,c_{a+1},z,x_1),(y,c_{a+1},c_{a+2},z,x_2),\ldots, (y,c_{k-2},c_{k-1},z,x_{k-a-1}),(y,c_{k-1},c_{k},z,x_{k-a})\};
\end{align*}
where $\mathcal{D}_1$ is understood to be empty and $c_0=c_k$ if $a=0$, and $\mathcal{D}_2$ is understood to be empty if $a=k$. Then $\mathcal{D}_1 \cup \mathcal{D}_2$ is a decomposition with the required properties.
\end{proof}

\begin{theorem}[\cite{BryHorPet14}]\label{AlspachEven}
Let $n$ be a positive even integer and let $m_1,\ldots, m_{t}$ be integers such that $3 \leq m_i \leq n$ for $i\in\{1,\ldots,t\}$ and $m_1+m_2+\dots+m_{t}= \binom{n}{2}-\frac{n}{2}$. Then there exists an $(m_1,\ldots, m_{t})$-decomposition of $K_n-I$, where $I$ is a $1$-factor with vertex set $V(K_n)$.
\end{theorem}

\begin{theorem}[{\cite{ChoFuHua99}}]\label{ChouFuHuang}
Let $a$, $b$, $p$, $q$ and $r$ be positive integers such that $a\geq 4$ and $b\geq 6$ are even. Then there exists a $(4^p,6^q,8^r)$-decomposition of $K_{a,b}$ if and only if $4p+6q+8r=ab$.
\end{theorem}

Theorem \ref{Bipartite} below is slightly stronger than Theorem 1.1 of \cite{Horsley12} but is easily proved using the following lemma which is Lemma 3.6 of that paper.

\begin{lemma}[\cite{Horsley12}]\label{BipartiteJoin}
Let $M$ be a list of integers and let $a$, $b$, $h$, $n$ and $n'$ be positive integers such that $a\leq b$, $n+n'\leq 3h$, $n+n'+h\leq 2a+2$ if $a<b$, and $n+n'+h\leq 2a$ if $a=b$. If there exists an $(M,h,n,n')$-decomposition of $K_{a,b}$, then there exists an $(M,h,n+n')$-decomposition of $K_{a,b}$.
\end{lemma}

\begin{theorem}\label{Bipartite}
Let $a$ and $b$ be positive integers such that $a$ and $b$ are even and $a\leq b$, and let $m_1,
m_2,\ldots, m_{\T}$ be even integers such that $4\leq m_1 \leq m_2 \leq\dots\leq m_{\T}$. If
\begin{itemize}
    \item[$(\rm{B}1)$]
$m_{\T}\leq 3m_{\T-1}$;
    \item[$(\rm{B}2)$]
$m_{\T-1}+m_{\T}\leq 2a+2$ if $a<b$ and $m_{\T-1}+m_{\T}\leq 2a$ if $a=b$; and
    \item[$(\rm{B}3)$]
$m_1+m_2+\dots+m_{\T}= ab$;
\end{itemize}
then there exists an $(m_1, m_2,\ldots, m_{\T})$-decomposition of $K_{a,b}$.
\end{theorem}

\begin{proof} Suppose for a contradiction that there exists a non-decreasing list of integers that satisfies the hypotheses of the theorem but for which there is no  corresponding decomposition of $K_{a,b}$, and amongst all such lists let $Z=z_1,\ldots,z_{\T}$ be one with a maximum number of entries.

It follows from Theorem \ref{ChouFuHuang} that $z_{\T} > 8$. Let $Z^\star$ be the list $z_1,z_2,\ldots,z_{\T-1},4,z_{\T}-4$ reordered so as to be non-decreasing. Since $Z$ satisfies the conditions of the claim, so must $Z^\star$, and since $Z^\star$ has more entries than $Z$, there exists a $(Z^\star)$-decomposition of $K_{a,b}$. However, by applying Lemma~\ref{BipartiteJoin} with $n=4$, $n'=z_{\T}-4$ and $h=z_{\T-1}$ we obtain a $(Z)$-decomposition of $K_{a,b}$ which is a contradiction.
\end{proof}

For each even integer $\ell \geq 4$ we define a list $R_{\ell}$ as follows
$$R_{\ell}=\left\{
  \begin{array}{ll}
    4^{\ell/4} & \hbox{if $\ell \equiv 0 \mod{4}$;} \\
    4^{(\ell-6)/4},6 & \hbox{if $\ell \equiv 2 \mod{4}$.}
  \end{array}
\right.$$
We also define $R_0$ to be the empty list. Given a list $R_{\ell}$ and a positive integer $i$ we define $R^i_{\ell}$ to be the list obtained by concatenating $i$ copies of $R_{\ell}$.

\begin{lemma}\label{BaseDecomposition}
Let $u, v$ and $m$ be odd integers such that  $m\geq 11$, $(u,v)$ is $m$-admissible, $v-u\geq m+1$, $u\geq m$ if $m\in\{11,13,15\}$, and $u\geq m-2$ if $m\geq 17$. Let $k$, $t$ and $x$ be the non-negative integers such that $u(v-u)=(m-1)k + t$, $t < m-1$ and $m(k+x)=\binom{v}{2}-\binom{u}{2}$. Then, for some $h \in \{4,6,\ldots,m-7\} \cup \{m-3\}$, there exists an $(m^x,3^k,h,R^{k-1}_{m-3},R^1_{m-h-3})$-decomposition of $K_v-K_u$ in which each cycle of length less than $m$ contains at most one pure edge.
\end{lemma}

\begin{proof} Observe that $k$ is the maximum number of pairwise edge-disjoint $m$-cycles in $K_v-K_u$ that each contain exactly one pure edge. Also note that $t$ is even and so $t \leq m-3$. Let $w=v-u$, note that $w$ is even, and let $p$ and $q$ be the non-negative integers such that $k=w(p+\frac{1}{2})+q$ and $q <w$. We will make use of the following facts throughout this proof.

\begin{equation}\label{edges}
uw =  (w(p+\tfrac{1}{2})+q)(m-1) + t
\end{equation}
\begin{equation}\label{claim}
2(u-2p) \geq
\left\{
  \begin{array}{ll}
    \frac{4}{3}m+\frac{22}{3} & \hbox{if $p = 0$} \\
    \frac{4}{3}m+\frac{34}{3} & \hbox{if $p \geq 1$}
  \end{array}
\right.
\end{equation}

Note that \eqref{edges} follows directly from the definitions of $k$, $t$, $p$ and $q$. To see that \eqref{claim} holds, observe that when $p=0$ and $m\in\{11,13,15\}$ we have $2(u-2p) = 2u \geq 2m$, which implies $2(u-2p) \geq \frac{4}{3}m+\frac{22}{3}$ since $m\geq 11$. When $p=0$ and $m\geq 17$ we have $2(u-2p) = 2u \geq 2m-4$, which implies $2(u-2p) \geq \frac{4}{3}m+\frac{22}{3}$ since $m\geq 17$. Also, \eqref{edges} implies $uw\geq w(p+\frac{1}{2})(m-1)$ and so $u\geq(p+\frac{1}{2})(m-1)$. Thus $2(u-2p) \geq 2p(m-3)+m-1$. So when $p \geq 1$ we have $2(u-2p) \geq 3m-7$, which implies $2(u-2p) \geq \frac{4}{3}m+\frac{34}{3}$ since $m\geq 11$.

Let $U=\{y_1,y_2,\ldots,y_u\}$ and $W=\{z_1,z_2,\ldots,z_w\}$ be disjoint sets of vertices. We will construct a decomposition of $K_{U\cup W}-K_U$ with the desired properties. Let $I$ be a $1$-factor with vertex set $W$. The proof divides into two cases depending on whether $t=0$ or $t>0$.

\noindent{\bf Case 1.} Suppose that $t>0$. Then $q>0$, for otherwise \eqref{edges} implies $w(u-(p+\frac{1}{2})(m-1))=t>0$ which contradicts the facts that $w>m>t$ and $u-(p+\frac{1}{2})(m-1)$ is an integer.

Depending on the value of $q$, we define integers $p'$, $q'$ and $q''$ so that
$w(p'+\frac{1}{2})+q'+q''=k$ according to the following table.
$$\begin{array}{|c|c |c| c|} \hline
 & p' & q' & q''  \\ \hline
q \in \{1,2,3,4\} & p-1& w  & q  \\
q \in \{5,7,\ldots,w-1\}& p & q-1 & 1 \\
q \in \{6,8,\ldots,w-2\}& p & q-2 & 2 \\ \hline
\end{array} $$
We show that $p' \geq 0$ by establishing that it cannot be the case that both $p=0$ and $q \in \{1,2,3,4\}$. If $p=0$ and $q \in \{1,2,3,4\}$, then \eqref{edges} implies $uw \leq (m-1)(\frac{w}{2}+4) + (m-3)$ and hence $w(u-\frac{m-1}{2}) \leq 5m-7$. Because $u\geq m$ if $m\in\{11,13,15\}$ and $u\geq m-2$ if $m\geq 17$, we have that $u-\frac{m-1}{2} \geq 6$ and we obtain a contradiction by noting that $w > m$.

We define $h$ to be the smallest integer in $\{4,6,\ldots,m-7\} \cup \{m-3\}$ such that $h\geq \frac{2q''+t}{3}$. Using the facts that $q'' \in \{1,2,3,4\}$ and $t \leq m-3$ it is routine to check that if $\frac{2q''+t}{3}>m-7$, then $m=11$, $q'' \in \{3,4\}$ and $h=8$. Thus $h$ is well-defined and, if it is not the case that $\frac{2q''+t}{3} \leq h \leq \frac{2q''+t+5}{3}$, then either $2q''+t \leq 6$ and $h=4$ or $m=11$, $q'' \in \{3,4\}$ and $h=8$. We claim that
\begin{equation}\label{cycleSumBound}
2q''+t+h \leq
\left\{
  \begin{array}{ll}
	\frac{4}{3}m+\frac{28}{3} & \hbox{if $p' = p-1$;} \\
    \frac{4}{3}m+3 & \hbox{if $p' = p$.}
  \end{array}
\right.
\end{equation}
To see that this is the case note that if $2q''+t \leq 6$ and $h=4$ or if $m=11$, $q'' \in \{3,4\}$ and $h=8$, then \eqref{cycleSumBound} holds (recall that $t \leq m-3$ and $m \geq 11$). Otherwise, $2q''+t+h \leq \frac{4}{3}(2q''+t)+\frac{5}{3}$ and hence $2q''+t+h \leq \frac{4}{3}m+\frac{1}{3}(8q''-7)$ using $t \leq m-3$. Because $q'' \leq 2$ if $p'=p$ and $q'' \leq 4$ if $p'=p-1$, \eqref{cycleSumBound} holds.

We will complete the proof by constructing an $(m^x,3^k,h,R^{k-1}_{m-3},R^1_{m-h-3})$-decomposition of $K_{U \cup W}-K_U$ in which each cycle of length less than $m$ contains at most one pure edge. We will construct this decomposition in such a way that the pure edges in the $3$-cycles of the decomposition form $p'$ $w$-cycles, a 1-factor with $w$ vertices, a $q'$-cycle, and a $q''$-path (recall that $p'w+w/2+q'+q''=k$). Our required decomposition can be obtained as
$$(\D_1\setminus \{H_1,\ldots,H_p,C',C'' \}) \cup (\D_2\setminus \{C^\dag\}) \cup \D_3 \cup \D_4 \cup \D_5$$
where $\D_1,\D_2,\D_3,\D_4,\D_5$ are given as follows.

\begin{itemize}
	\item
$\D_1$ is a $(w^{p'},m^{x-1},q',m+q''-t)$-decomposition of $K_W-I$, that includes $p'$ $w$-cycles $H_1,\ldots,H_{p'}$, a $q'$-cycle $C'$, and an $(m+q''-t)$-cycle $C''$ containing the path $[z_1,z_2,\ldots,z_{q''+1}]$ and not containing the $\frac{t}{2}-1$ vertices $z_{q''+2},z_{q''+3},\ldots,z_{q''+\frac{t}{2}}$. A $(w^{p'},m^{x-1},q',m+q''-t)$-decomposition of $K_W-I$ exists by Theorem \ref{AlspachEven} because $mx+p'w+q'+q''-t=mx+k-t-\frac{w}{2}=\binom{w}{2}-\frac{w}{2}$ (note that the definitions of $k$, $t$ and $x$ imply that $\binom{w}{2}=mx+k-t$). We can relabel the vertices of this decomposition to ensure that $C''$ has the specified properties because $(m+q''-t)+(\frac{t}{2}-1)=m-\frac{t}{2}+q''-1 \leq w$. (If $m-\frac{t}{2}+q''-1 \geq w+1$, then $t\geq 2$, $q''\leq 4$ and $w\geq m+1$ imply
$(t,q'',w)=(2,4,m+1)$ and hence $q=4$. However, in this case \eqref{edges} implies that $u(m+1)=(m+1)(m-1)(p+\frac{1}{2})+4m-2$ and hence that $(m+1)$ divides $4m-2$. This contradicts $m\geq 11$.)
	\item
$\D_2$ is a $(2q''+t,h,R^{k-1}_{m-3},R^1_{m-h-3})$-decomposition of $K_{\{y_1,\ldots,y_{u-2p'-1}\},W} - K_{\{y_{u-2p'-2},y_{u-2p'-1}\},V(C')}$ that includes the
$(2q''+t)$-cycle $C^{\dag}=(y_1,z_1,y_2,z_2,\ldots,y_{q''+\frac{t}{2}},z_{q''+\frac{t}{2}})$. We form $\D_2$ by first decomposing $K_{\{y_{u-2p'-2},y_{u-2p'-1}\},W\setminus V(C')}$ into $\frac{w-q'}{2}$ $4$-cycles (if $q'<w$) and then decomposing $K_{\{y_1,y_2,\ldots,y_{u-2p'-3}\}, W}$ into cycles of the remaining lengths. Note that $R^{k-1}_{m-3}$ contains at least $k-1$ $4$'s since $m \geq 11$, and also that $k-1=w(p+\frac{1}{2})+q-1 \geq \frac{w}{2}+q-1  \geq \frac{w-q'}{2}$.

A decomposition of $K_{\{y_1,y_2,\ldots,y_{u-2p'-3}\}, W}$ into cycles of the remaining lengths exists by Theorem \ref{Bipartite}. Let $m_\tau$ and $m_{\tau-1}$ be respectively the greatest and second greatest of the remaining cycle lengths. To see that (B1), (B2) and (B3) hold, we first suppose that $(m_\tau,m_{\tau-1})=(2q''+t,h)$. It follows from the definition of $h$ that $2q''+t\leq 3h$ and hence that (B1) holds. If $u-2p'-3 \geq w$, then (B2) holds because,  using $m \geq 11$ and \eqref{cycleSumBound}, we have
$$2w \geq 2(m+1) \geq \tfrac{4}{3}m+\tfrac{28}{3} \geq 2q''+t+h.$$
If $u-2p'-3 < w$ and $p'=p$, then (B2) holds because, using \eqref{claim} and \eqref{cycleSumBound}, we have
$$2(u-2p'-3)+2 = 2(u-2p)-4 \geq (\tfrac{4}{3}m+\tfrac{22}{3})-4 > \tfrac{4}{3}m+3 \geq 2q''+t+h.$$ If $u-2p'-3<w$ and $p'=p-1$, then $p \geq 1$ and (B2) holds because, using \eqref{claim} and \eqref{cycleSumBound}, we have
$$2(u-2p'-3)+2 = 2(u-2p) \geq \tfrac{4}{3}m+\tfrac{34}{3} > \tfrac{4}{3}m+\tfrac{28}{3} \geq 2q''+t+h.$$
Finally, (B3) holds because, using the definitions of $t$, $k$, $p$ and $q$, and the fact that $q'+q''-q=w(p-p')$, we have
\begin{align*}
  (k-1)(m-3) +(m-h-3)+h+(2q''+t)-4\tfrac{w-q'}{2} &= k(m-3)+t-2w+2(q'+q'') \\
  &=(uw-2k)-2w+2(q'+q'') \\
  &=(u-2p-3)w+2(q'+q''-q) \\
  &=(u-2p'-3)w.
\end{align*}
Now suppose that $(m_\tau,m_{\tau-1})\neq(2q''+t,h)$. Then it must be that $m_\tau \leq 12$ and $m_{\tau-1} \in \{4,6\}$. Clearly (B1) holds. Also, $m_\tau+m_{\tau-1} \leq 18$. So if $w\leq u-2p'-3$ then (B2) holds because $2w \geq 2(m+1) \geq 24$, and if $w> u-2p'-3$ then (B2) holds because \eqref{claim} implies $2(u-2p'-3)+2 \geq 18$ (note that $p' \leq p$ and $m \geq 11$).
Finally, (B3) holds by the argument above. We can relabel the vertices of this decomposition to ensure that $C^{\dag}=(y_1,z_1,y_2,z_2,\ldots,y_{q''+\frac{t}{2}},z_{q''+\frac{t}{2}})$.
	\item
$\D_3$ is a $(3^{q'})$-decomposition of $K^c_{\{y_{u-2p'-2},y_{u-2p'-1}\}}\vee C'$ which exists by Lemma~\ref{MixedCycles}.
	\item
$\D_4$ is a $(3^{p'w+w/2})$-decomposition of $K_{\{y_{u-2p'},\ldots,y_{u}\}}^c \vee (I\cup H_1 \cup \cdots \cup H_{p'})$ which exists by applying Lemma~\ref{MixedCycles} to $K_{\{y_{u-2p'-1+i},y_{u-p'-1+i}\}}^c\vee H_i$ for $i\in\{1,2,\ldots,p'\}$ and taking the obvious decomposition of $K_{\{y_u\}} \vee I$.
	\item
$\D_5$ is the $(3^{q''},m)$-decomposition of $C'' \cup C^{\dag}$ given by
\begin{multline*}
\{(z_1,y_2,z_2),(z_2,y_3,z_3),\ldots,(z_{q''},y_{q''+1},z_{q''+1})\}\; \cup \\ \{(C''-[z_1,z_2,\ldots,z_{q''+1}])\cup [z_{q''+1},y_{q''+2},z_{q''+2},y_{q''+3},\ldots,y_{q''+\frac{t}{2}},z_{q''+\frac{t}{2}},y_1,z_1]\}.
\end{multline*}
\end{itemize}

\noindent{\bf Case 2.} Suppose that $t=0$. Then \eqref{edges} reduces to $uw=(w(p+\tfrac{1}{2})+q)(m-1)$. Note that $q\neq 1$, since if $q=1$ then $uw=w(p+\frac{1}{2})(m-1) + m-1$ and so $w$ divides $m-1$ which contradicts $w\geq m+1$. Depending on the value of $q$, we define integers $p'$, $q'$ and $q''$ so that $w(p'+\frac{1}{2})+q'+q''=k$ according to the following table.

$$\begin{array}{|c|c |c| c|} \hline
 & p' & q' & q''  \\ \hline
q \in \{0,3,5\} & p-1& w  & q  \\
q=2 & p-1& w-2 & 4 \\
q \in \{4,6,\ldots,w-2\}& p& q & 0 \\
q \in \{7,9,\ldots,w-1\} & p& q-3 & 3\\ \hline
\end{array} $$
We show that $p' \geq 0$ by establishing that it cannot be the case that both $p=0$ and $q \in \{0,2,3,5\}$. If $p=0$ and $q \in \{0,2,3,5\}$, then \eqref{edges} implies $uw \leq (m-1)(\frac{w}{2}+5)$ and hence $w(u-\frac{m-1}{2}) \leq 5m-5$. Because $u\geq m$ if $m\in\{11,13,15\}$ and $u\geq m-2$ if $m\geq 17$, we have that $u-\frac{m-1}{2} \geq 6$ and we obtain a contradiction by noting that $w > m$.

We will complete the proof by constructing an $(m^x,3^k,R^k_{m-3})$-decomposition of $K_{U \cup W}-K_U$ in which each cycle of length less than $m$ contains at most one pure edge. We will construct this decomposition in such a way that the pure edges in the $3$-cycles of the decomposition form $p'$ $w$-cycles, a 1-factor with $w$ vertices, a $q'$-cycle, and a $q''$-cycle if $q'' \neq 0$ (recall that $p'w+w/2+q'+q''=k$). The desired decomposition can be obtained as
$$(\D_1\setminus S_1) \cup (\D_2\setminus S_2) \cup \D_3 \cup \D_4 \cup \D_5$$
where $(S_1,S_2)=(\{H_1,H_2,\ldots,H_{p'},C',C''\},\{C^\dag\})$ if $q''\neq 0$, $(S_1,S_2)=(\{H_1,H_2,\ldots,H_{p'},C'\},\emptyset)$ if $q''=0$, and $\D_1,\D_2,\D_3,\D_4,\D_5$ are given as follows.

\begin{itemize}
	\item $\D_1$ is an $(m^x,w^{p'},q',q'')$-decomposition of $K_W-I$ that includes $p'$ $w$-cycles $H_1,\ldots,H_{p'}$, a $q'$-cycle $C'$ and, if $q''\neq 0$, the $q''$-cycle $C''=(z_1,z_2,\ldots,z_{q''})$. An $(m^x,w^{p'},q',q'')$-decomposition of $K_W-I$ exists by Theorem \ref{AlspachEven} because $mx+p'w+q'+q''=mx+k-\frac{w}{2}=\binom{w}{2}-\frac{w}{2}$ (note that the definitions of $k$, $t$ and $x$ imply that $\binom{w}{2}=mx+k-t$). We can relabel the vertices of this decomposition to ensure that $C''=(z_1,z_2,\ldots,z_{q''})$.
	\item
$\D_2$ is a $(2q'',3^k,R^k_{m-3})$-decomposition of $K_{\{y_1,y_2,\ldots,y_{u-2p'-1}\}, W}-
K_{\{y_{u-2p'-2},y_{u-2p'-1}\},V(C')}$ that includes the $(2q'')$-cycle $C^\dag=(y_1,z_1,y_2,z_2,\ldots,y_{q''},z_{q''})$ if $q''\neq 0$.

We form $\D_2$ by first decomposing $K_{\{y_{u-2p'-2},y_{u-2p'-1}\},W\setminus V(C')}$ into $\frac{w-q'}{2}$ $4$-cycles (if $q'<w$) and then decomposing $K_{\{y_1,y_2,\ldots,y_{u-2p'-3}\}, W}$ into cycles of the remaining lengths. Note that $R^k_{m-3}$ contains at least $k$ $4$'s since $m \geq 11$, and also that $k=w(p+\frac{1}{2})+q \geq \frac{w}{2}+q  \geq \frac{w-q'}{2}$.

A decomposition of $K_{\{y_1,y_2,\ldots,y_{u-2p'-3}\}, W}$ into cycles of the remaining lengths exists by Theorem \ref{Bipartite}. Let $m_\tau$ and $m_{\tau-1}$ be respectively the greatest and second greatest of the remaining cycle lengths. Note that $m_\tau \leq \max(2q'',6) \leq 10$ and $m_{\tau-1} \in \{4,6\}$. Clearly (B1) holds. If $w\leq u-2p'-3$, then (B2) holds because $m_{\tau-1}+m_\tau \leq 16$ and  $2w \geq 2(m+1) \geq 24$. If $w> u-2p'-3$, then (B2) holds because $m_{\tau-1}+m_\tau \leq 16$ and \eqref{claim} implies that $2(u-2p'-3)+2 \geq 18$ (note that $p' \leq p$ and $m \geq 11$). Finally, (B3) holds by a similar argument to that used in Case 1. We can relabel the vertices of this decomposition to ensure that $C^\dag=(y_1,z_1,y_2,z_2,\ldots,y_{q''},z_{q''})$ if $q''\neq 0$.
	\item
$\D_3$ is a $(3^{q'})$-decomposition of $K_{\{y_{u-2p'-2},y_{u-2p'-1}\}}^c\vee C'$ which exists by Lemma~\ref{MixedCycles}.
	\item
$\D_4$ is a $(3^{p'w+w/2})$-decomposition of $K_{\{y_{u-2p'},\ldots,y_{u}\}}^c \vee (I\cup H_1 \cup \cdots \cup H_{p'})$ which exists by applying Lemma~\ref{MixedCycles} to
$K_{\{y_{u-2p'-1+i},y_{u-p'-1+i}\}}^c\vee H_i$ for $i\in\{1,2,\ldots,p'\}$ and taking the obvious decomposition of $K_{\{y_u\}}\vee I$.
	\item
$\D_5$ is the $(3^{q''})$-decomposition of $C''\cup C^\dag$ given by $\{(z_1,y_2,z_2),(z_2,y_3,z_3),\ldots,(z_{q''-1},y_{q''},z_{q''}),(z_{q''},y_1,z_1)\}$ if $q'' \neq 0$, and $\D_5=\emptyset$ if $q''=0$.\qedhere
\end{itemize}
\end{proof}

\begin{lemma}\label{BaseDecomposition_9}
Let $u$ and $v$ be positive integers such that $(u,v)$ is $9$-admissible, $v-u\geq 10$ and $u \geq 9$. Let $k$, $t$ and $x$ be the non-negative integers such that $u(v-u)=8k + t$, $t <8$ and $9(x+k)=\binom{v}{2}-\binom{u}{2}$.
Then, for some non-negative integer $k'\leq k$, there exists a $(3^{k-k'},4^{k'},5^{k'},6^{k-k'},9^x)$-decomposition of $K_v-K_u$ in which each cycle of length less than $9$ contains at most one pure edge.
\end{lemma}

\begin{proof} Observe that $k$ is the maximum number of pairwise edge-disjoint $9$-cycles in $K_v-K_u$ that each contain exactly one pure edge. Also note that $t$ is even and so $t \leq 6$.
Let $w=v-u$, note that $w$ is even, and let $p$ and $q$ be the non-negative integers such that $k=(p+1)\frac{w}{2} +q$ and $q <\frac{w}{2}$. We will make use of the following fact, which follows directly from the definitions of $k$, $t$, $p$ and $q$, throughout this proof.

\begin{equation}\label{edges2}
uw= 8((p+1)\tfrac{w}{2} +q)+t
\end{equation}
Note that $p\geq 1$, for otherwise $uw\leq 4w+8q+t\leq 4w +8(\frac{w}{2}-1)+6<8w$ which contradicts $u\geq 9$.
Also note that $u\geq 4p+5$, because $u$ is odd and \eqref{edges2} implies $uw\geq 4(p+1)w$. From this, we can see that
\begin{equation}\label{sideBound}
u-2p-3 \geq 2p+2 \geq 4.
\end{equation}

Let $U=\{y_1,y_2,\ldots,y_u\}$ and $W=\{z_1,z_2,\ldots,z_w\}$ be disjoint sets of vertices. We will construct a decomposition of $K_{U\cup W}-K_U$ with the desired properties. Let $I$ be a $1$-factor with vertex set $W$. The proof divides into two cases depending on whether $t=0$ or $t>0$.

\noindent{\bf Case 1.} Suppose that $t>0$. Then $q>0$, for otherwise \eqref{edges2} implies  $w(u-4p-4)=t>0$ which contradicts the facts that $w\geq 10>t$ and $u-4p-4$ is an integer. Depending on the value of $q$, we define integers $p'$, $q'_3$, $q'_5$ and $q''$ so that
$(p'+1)\frac{w}{2}+q'_3+q'_5+q''=k$ according to the following table.

$$\begin{array}{|c|c|c|c|c|} \hline
	& p' & q'_3 & q'_5 & q''  \\ \hline
q\in\{1,2,3\} & p-1& 2q-2 & \frac{w}{2}+1-q  & 1  \\
q\in\{4,\ldots,\frac{w}{2}-1\} & p & 0 & q-1 & 1 \\ \hline
\end{array} $$

We will complete the proof by constructing a $(3^{w/2+q'_3+q''},4^{p'w/2+q'_5},5^{p'w/2+q'_5},6^{w/2+q'_3+q''},9^x)$-decomposition of $K_{U \cup W}-K_U$ in which each cycle of length less than $9$ contains at most one pure edge. We will construct this decomposition in such a way that the pure edges in the $3$-cycles and $5$-cycles of the decomposition form $p'$ $\frac{w}{2}$-cycles, a 1-factor with $w$ vertices, a $(q_3'+q_5')$-cycle, and a $q''$-path (recall that $(p'+1)\frac{w}{2}+q'_3+q'_5+q''=k$). The desired decomposition can be obtained as
$$(\D_1\setminus \{H_1,H_2,\ldots,H_{p'},C',C''\}) \cup (\D_2\setminus \{C^\dag\})\cup \D_3\cup \D_4\cup \D_5$$
where $\D_1$, $\D_2$, $\D_3$, $\D_4$ and $\D_5$ are given as follows.

\begin{itemize}
	\item
$\D_1$ is a $(9^{x-1},(\frac{w}{2})^{p'},q'_3+q'_5,9+q''-t)$-decomposition of $K_W-I$ that includes $p'$ $(\frac{w}{2})$-cycles $H_1,H_2,\ldots,H_{p'}$, a $(q'_3+q'_5)$-cycle $C'$, and a $(9+q''-t)$-cycle $C''$ containing the edge $z_1z_2$ and not containing the $\frac{t}{2}-1$ vertices $z_{3},z_{4},\ldots,z_{\frac{t}{2}+1}$ if $t>2$.
A $(9^{x-1},(\frac{w}{2})^{p'},q'_3+q'_5,9+q''-t)$-decomposition of $K_W-I$ exists by Theorem \ref{AlspachEven} because $9(x-1)+p'\frac{w}{2}+q'_3+q'_5+9+q''-t=9x+k-t-\frac{w}{2}=\binom{w}{2}-\frac{w}{2}$ (note that the definitions of $k$, $t$ and $x$ imply that $\binom{w}{2}=9x+k-t$). We can relabel the vertices of this decomposition to ensure that $C''$ has the specified properties because $(9+q''-t)+(\frac{t}{2}-1)=9-\frac{t}{2}< 10\leq w$.
	\item
$\D_2$ is a $(4^{p'w/2+q'_5},6^{w/2+q'_3+q''},t+2q'')$-decomposition of $K_{\{y_1,y_2,\ldots,y_{u-2p'-1}\}, W} - K_{\{y_{u-2p'-2},y_{u-2p'-1}\},V(C')\cup Q}$, where $Q$ is a subset of $W\setminus V(C')$ of size $q'_5$, that includes the $(t+2q'')$-cycle $C^\dag=(y_1,z_1,y_2,z_2,\ldots,y_{\frac{t}{2}+1},z_{\frac{t}{2}+1})$. Note that $|W\setminus V(C')| = w-q'_3-q'_5 \geq q'_5$ follows from our choice of $q'_3$ and $q'_5$.

We form $\D_2$ by first decomposing $K_{\{y_{u-2p'-2},y_{u-2p'-1}\},W\setminus(V(C')\cup Q)}$ into $\frac{w-q'_3-2q'_5}{2}$ $4$-cycles and then decomposing $K_{\{y_1,y_2,\ldots,y_{u-2p'-3}\}, W}$ into cycles of the remaining lengths. Note that we desire $p'\frac{w}{2}+q'_5$ $4$-cycles in $\D_2$ and that $p'\frac{w}{2}+q'_5\geq \frac{w-q'_3-2q'_5}{2}$ follows from our choice of $p'$, $q'_3$ and $q'_5$ (recall that $p \geq 1$).

A decomposition of $K_{\{y_1,y_2,\ldots,y_{u-2p'-3}\}, W}$ into cycles of the remaining lengths exists by Theorem \ref{ChouFuHuang} because $t+2q''\in\{4,6,8\}$. To see that the conditions of Theorem \ref{ChouFuHuang} hold, note that $u-2p'-3\geq u-2p-3\geq 4$ using \eqref{sideBound},  that $w\geq 10$, and that
\begin{align*}
  &4(p'\tfrac{w}{2}+q_5'-\tfrac{w-q'_3-2q'_5}{2})+6(\tfrac{w}{2}+q_3'+q'')+(t+2q'')\\
  ={}&8((p'+1)\tfrac{w}{2}+q_3'+q'_5+q'')+t-(2p'+3)w \\
  ={}&uw-(2p'+3)w \\
  ={}&(u-2p'-3)w.
\end{align*}
We can relabel the vertices of this decomposition to ensure that $C^\dag=(y_1,z_1,y_2,z_2,\ldots,y_{\frac{t}{2}+1},z_{\frac{t}{2}+1})$.
	\item
$\D_3$ is a $(3^{q'_3},5^{q'_5})$-decomposition of $K_{\{y_{u-2p'-2},y_{u-2p'-1}\}}^c\vee (C'\cup K_Q^c)$ which exists by Lemma~\ref{MixedCycles}.
	\item
$\D_4$ is a $(3^{w/2},5^{p'w/2})$-decomposition of $K_{\{y_{u-2p'},\ldots,y_{u}\},W} \cup I \cup H_1 \cup \cdots \cup H_{p'}$ which exists by applying Lemma~\ref{MixedCycles} (with $a=0$ and $n=\frac{w}{2}$) to
$K_{\{y_{u-2p'-1+i},y_{u-p'-1+i}\}}^c\vee (H_i\cup K_{W\setminus V(H_i)}^c)$ for $i\in\{1,2,\ldots,p'\}$ and taking the obvious decomposition of $K_{\{y_u\}}\vee I$.
	\item
$\D_5$ is the following $(3^{q''},9^1)$-decomposition of $C''\cup C^\dag$ (recall that $q''=1$). $$\{(z_1,y_2,z_2)\} \cup \{(C''-[z_1,z_2])\cup [z_{2},y_{3},z_{3},y_{4},\ldots,y_{\frac{t}{2}+1},z_{\frac{t}{2}+1},y_1,z_1] \}$$

\end{itemize}

\noindent{\bf Case 2.} Suppose that $t=0$. Then \eqref{edges2} reduces to $uw=8((p+1)\frac{w}{2} +q)$.
Depending on the value of $q$, we define integers $p'$, $q'$ and $q''$ so that
$(p'+1)\frac{w}{2}+q'+q''=k$ according to the following table.

$$\begin{array}{|c|c |c| c|} \hline
 & p' & q' & q''  \\ \hline
q=0 & p-1& \frac{w}{2}  & 0  \\
q\in\{1,2\} & p-1& \frac{w}{2}-3+q  & 3   \\
q\in\{3,4,\ldots,\frac{w}{2}-1\} & p & q & 0\\ \hline
\end{array} $$

We will complete the proof by constructing a $(3^{w/2+q''},4^{p'w/2+q'},5^{p'w/2+q'},6^{w/2+q''},9^x)$-decomposition of $K_{U \cup W}-K_U$ in which each cycle of length less than $9$ contains at most one pure edge. We will construct this decomposition in such a way that the pure edges in the $3$-cycles and $5$-cycles of the decomposition form $p'$ $\frac{w}{2}$-cycles, a 1-factor with $w$ vertices, a $q'$-cycle, and a $q''$-cycle if $q'' \neq 0$ (recall that $(p'+1)\frac{w}{2}+q'+q''=k$). The desired decomposition can be obtained as
$$(\D_1\setminus S_1) \cup (\D_2\setminus S_2)\cup \D_3\cup \D_4\cup \D_5$$
where $(S_1,S_2)=(\{H_1,H_2,\ldots,H_{p'},C',C''\},\{C^\dag\})$ if $q''\neq 0$, $(S_1,S_2)=(\{H_1,H_2,\ldots,H_{p'},C'\}, \emptyset)$ if $q''= 0$, and
$\D_1$, $\D_2$, $\D_3$, $\D_4$ and $\D_5$ are given as follows.
\begin{itemize}
	\item
$\D_1$ is a $(9^x,(\frac{w}{2})^{p'},q',q'')$-decomposition of $K_W-I$ that includes $p'$ $(\frac{w}{2})$-cycles $H_1,H_2,\ldots,H_{p'}$, a $q'$-cycle $C'$ and a $q''$-cycle $C''=(z_1,z_2,\ldots,z_{q''})$ if $q''\neq 0$. A $(9^x,(\frac{w}{2})^{p'},q',q'')$-decomposition of $K_W-I$ exists by Theorem \ref{AlspachEven} because $9x+p'\frac{w}{2}+q'+q''=9x+k-\frac{w}{2}=\binom{w}{2}-\frac{w}{2}$ (note that the definitions of $k$, $t$ and $x$ imply that $\binom{w}{2}=9x+k-t$). We can relabel the vertices of this decomposition to ensure that $C''=(z_1,z_2,\ldots,z_{q''})$ if $q''\neq 0$.
	\item	
$\D_2$ is a $(4^{p'w/2+q'},6^{w/2+q''},2q'')$-decomposition of $K_{\{y_1,y_2,\ldots,y_{u-2p'-1}\}, W} - K_{\{y_{u-2p'-2},y_{u-2p'-1}\},V(C')\cup Q}$, where $Q$ is a subset of $W\setminus V(C')$ of size $q'$, that includes the $(2q'')$-cycle $C^\dag=(y_1,z_1,y_2,z_2,\ldots,y_{q''},z_{q''})$ if $q''\neq 0$. Note that $|W\setminus V(C')| = w-q' \geq q'$ follows from our choice of $q'$.

We form $\D_2$ by first decomposing $K_{\{y_{u-2p'-2},y_{u-2p'-1}\},W\setminus(V(C')\cup Q)}$ into $\frac{w}{2}-q'$ $4$-cycles (if $q'<\frac{w}{2}$) and then decomposing $K_{\{y_1,y_2,\ldots,y_{u-2p'-3}\}, W}$ into cycles of the remaining lengths.
Note that we desire $p'\frac{w}{2}+q'$ $4$-cycles in $\D_2$ and that $p'\frac{w}{2}+q'\geq \frac{w}{2}-q'$ follows from our choice of $p'$, $q'$ and $q''$ (recall that $p \geq 1$ and $w \geq 10$).

A decomposition of $K_{\{y_1,y_2,\ldots,y_{u-2p'-3}\}, W}$ into cycles of the remaining lengths exists by Theorem \ref{ChouFuHuang} because $2q''\in \{0,6\}$.
The conditions of Theorem \ref{ChouFuHuang} can be shown to hold by a similar argument to that used in Case 1. We can relabel the vertices of this decomposition to ensure that $C^{\dag}=(y_1,z_1,y_2,z_2,\ldots,y_{q''},z_{q''})$ if $q''\neq 0$.

	\item
$\D_3$ is a $(5^{q'})$-decomposition of $K_{\{y_{u-2p'-2},y_{u-2p'-1}\}}^c\vee (C'\cup K_{Q}^c)$ which exists by Lemma~\ref{MixedCycles}.

	\item
$\D_4$ is a $(3^{w/2},5^{p'w/2})$-decomposition of $K_{\{y_{u-2p'},\ldots,y_{u}\},W} \cup I \cup H_1 \cup \cdots \cup H_{p'}$ which exists by applying Lemma~\ref{MixedCycles} (with $a=0$ and $n=\frac{w}{2}$) to
$K_{\{y_{u-2p'-1+i},y_{u-p'-1+i}\}}^c\vee (H_i\cup K_{W\setminus V(H_i)}^c)$ for $i\in\{1,2,\ldots,p'\}$ and taking the obvious decomposition of $K_{\{y_u\}}\vee I$.
	\item
$\D_5$ is the $(3^{q''})$-decomposition of $C''\cup C^\dag$ given by $\{(z_1,y_2,z_2), (z_2,y_3,z_3), \ldots ,(z_{q''-1},y_{q''},z_{q''}), (z_{q''},y_1,z_1)\}$ if $q'' \neq 0$, and $\D_5=\emptyset$ if $q''=0$. \qedhere
\end{itemize}
 \end{proof}

\section{Proof of main results}\label{ProofSec}

We first prove the following lemma, which does most of the work toward proving Theorem \ref{MainTheorem}.

\begin{lemma}\label{AmostEverythingLemma}
Let $m$, $u$ and $v$ be positive odd integers such that $m \geq 9$, $v-u \geq m+1$ and $(u,v)$ is $m$-admissible. If either $m \geq 17$ and $u \geq m-2$ or $m \in \{9,11,13,15\}$ and $u \geq m$, then there exists an $m$-cycle decomposition of $K_v-K_u$.
\end{lemma}

\begin{proof} Let $k$ and $x$ be the integers such that $k=\lfloor\frac{u(v-u)}{m-1}\rfloor$ and $m(x+k)=\binom{v}{2}-\binom{u}{2}$. From Lemmas \ref{BaseDecomposition} and \ref{BaseDecomposition_9}, it follows that there is an $(m^x,M_1,\ldots,M_k)$-decomposition $\mathcal{D}_0$ of $K_v-K_u$ in which each cycle of length less than $m$ contains at most one pure edge, where for $j \in \{1,\ldots,k\}$ $M_j$ is a list of at least two integers with sum $m$ that contains exactly one odd integer.

We will now construct a sequence $\mathcal{D}_1,\ldots,\mathcal{D}_{k-1}$ such that, for each $i \in \{1,\ldots,k-1\}$, $\mathcal{D}_i$ is an $(m^{x+i+1},M_1,\ldots,M_{k-i-1})$-decomposition of $K_v-K_u$ such that, with the exception of $x$ $m$-cycles, each cycle in $\mathcal{D}_i$ contains at most one pure edge. This will suffice to complete the proof because $\mathcal{D}_{k-1}$ will be an $m$-cycle decomposition of $K_v-K_u$.

Let $\mathcal{D}_1$ be the decomposition obtained by applying Lemma \ref{JoiningLemma} to the packing $\mathcal{P}_0=\mathcal{D}_0 \setminus \mathcal{C}_0$, where $\mathcal{C}_0$ is a set of cycles in $\mathcal{D}_0$ with lengths given by the list $M_{k-1},M_k$.
For each $i \in \{1,\ldots,k-2\}$, let $\mathcal{D}_{i+1}$ be the decomposition obtained by applying Lemma \ref{JoiningLemma} to the packing $\mathcal{P}_i=\mathcal{D}_i \setminus (\{C_i\} \cup \mathcal{C}_i)$, where $C_i$ is an $m$-cycle in $\mathcal{D}_i$ that contains exactly one pure edge and $\mathcal{C}_i$ is a set of cycles in $\mathcal{D}_i$ with lengths given by the list $M_{k-i-1}$. For each $i \in \{0,\ldots,k-2\}$, it is easy to verify that $\mathcal{D}_{i+1}$ has the required properties because $\mathcal{D}_{i}$ does and, by Lemma \ref{JoiningLemma}, $\mathcal{D}_{i+1}\setminus \{C',C''\}$ is a repacking of $\mathcal{P}_i$ for some distinct $m$-cycles $C',C'' \in \mathcal{D}_{i+1}$.
\end{proof}

The following lemma exploits the fact that two $m$-cycle decompositions of complete graphs with holes can be  ``nested'' to create another. It will be used to help deal with the remaining cases of Theorem \ref{MainTheorem} and in the proof of Corollary \ref{MissingCasesTheorem}.

\begin{lemma}\label{ReductionLemma}
Let $m \geq 9$ and $u$ be odd integers. If there exists an $m$-cycle
decomposition of $K_{u^{\star}}-K_u$ for some positive integer $u^{\star} \geq m$, then there exists an $m$-cycle decomposition of $K_v-K_u$ for each integer $v$ such that $v \geq u^{\star}+m+1$ and $(u,v)$ is $m$-admissible.
\end{lemma}

\begin{proof} Let $v$ be an integer such that $v \geq u^{\star}+m+1$ and $(u,v)$ is
$m$-admissible. Let $U$, $U^{\star}$ and $V$ be sets such that $|U|=u$, $|U^{\star}|=u^{\star}$, $|V|=v$, and $U \subseteq U^{\star} \subseteq V$. By our hypotheses, there exists an $m$-cycle decomposition $\D_1$ of $K_{U^{\star}}-K_U$. Now note that $v-u^{\star}\geq m+1$ and $u^{\star}\geq m$ from our hypotheses, and that $(u^{\star},v)$ is $m$-admissible because $(u,u^{\star})$ and $(u,v)$ are $m$-admissible. Thus, by Lemma \ref{AmostEverythingLemma}, there is an $m$-cycle decomposition $\D_2$ of $K_V-K_{U^{\star}}$. Then $\D_1 \cup \D_2$ is an $m$-cycle decomposition of $K_V-K_{U}$.
\end{proof}

\begin{proof}[{\bf Proof of Theorem \ref{MainTheorem}.}]
We first note that Theorem~\ref{MainTheorem} is equivalent to showing that for $u \geq m-2$ and $v-u\geq m+1$, there exists an $m$-cycle decomposition of $K_v-K_u$ if and only if $(u,v)$ is $m$-admissible (note that $v-u \geq m+1$ guarantees that $(v-m)(v-1) \geq u(u-1)$). By Lemma \ref{NecConds}, if there exists an $m$-cycle decomposition of $K_v-K_u$, then $(u,v)$ is $m$-admissible. So it is sufficient to prove that for any $m$-admissible pair $(u,v)$ of integers such that $u \geq m-2$ and $v-u\geq m+1$, there exists an $m$-cycle decomposition of $K_v-K_u$. This is established for $m\leq 7$ (see \cite{BryHofRod96, BryRodSpi97, MenRos83}), so we can suppose that $m\geq 9$. By Lemma \ref{AmostEverythingLemma} there exists an $m$-cycle decomposition of $K_v-K_u$ if either $m\geq 17$ and $u \geq m-2$ or $m \in \{9,11,13,15\}$ and $u \geq m$, so we can further suppose that $m \in \{9,11,13,15\}$ and $u=m-2$.

Because $u=m-2$ and $m \in \{9,11,13,15\}$, it follows from $v-u \geq m+1$ that $v \geq 2m-1$. Furthermore, it is routine to check that (N1) and (N2) of Lemma \ref{NecConds} imply that $v \equiv 3 \mod{2m}$ or $v \equiv m-2 \mod{2m}$. Provided that there exists an $m$-cycle decomposition of $K_{2m+3}-K_{m-2}$, by Lemma \ref{ReductionLemma} there exists an $m$-cycle decomposition of $K_v-K_{m-2}$ if $v \geq 4m+3$. Thus, because $v \geq 2m-1$, it suffices to show that there is an $m$-cycle decomposition of $K_{2m+3}-K_{m-2}$ and of $K_{3m-2}-K_{m-2}$. For each $m \in \{9,11,13,15\}$, we have found these two decompositions using a computer program that implements basic cycle switching techniques to augment decompositions. These decompositions are included as supplementary material with electronic versions of this paper.
\end{proof}

\begin{proof}[{\bf Proof of Corollary \ref{MissingCasesTheorem}.}] Part (i) follows from
Lemma~\ref{ReductionLemma} (with $u^{\star}=\nu_m(u)$). Part (ii) follows from Theorem~\ref{MainTheorem}. For $u=1$ part (iii) follows from Theorem \ref{CSExist} and for $u=3$ it follows by removing a $3$-cycle from a decomposition of a complete graph into $m$-cycles and a single $3$-cycle which exists by the main result of \cite{BryHorPet14}. If $u>3$, then part (iii) follows from Theorem~\ref{MainTheorem}, noting that $u > \frac{(m-1)(m-2)}{2}$ implies that $\frac{u(m+1)}{m-1}+1>u+m-1$.
\end{proof}

Now we shall prove Theorem \ref{DWTheorem}. Theorem \ref{MainTheorem} can be shown to cover the exceptions to the following result from \cite{BryRod94_m}, and as a consequence we can completely solve the embedding problem for $m$-cycle systems in the case where $m$ is an odd prime power.

\begin{theorem}[\cite{BryRod94_m}]\label{DWExceptions}
Let $m$, $u$ and $v$ be positive integers such that $m$ is odd, $u<v$, and $u,v\equiv 1$ or $m\mod{2m}$. An $m$-cycle system of order $u$ can be embedded in an $m$-cycle system of order $v$ if and only if $v\geq \frac{(m+1)u}{m-1}+1$, except sometimes when $u\equiv v\equiv m\mod{2m}$ and $\frac{(m+1)u}{m-1}+1 \leq v \leq \frac{(m+1)u}{m-1} +2m$.
\end{theorem}

\begin{lemma}\label{DWprimes}
Let $m$ be an odd prime power. For positive integers $u$ and $v$ with $u<v$, an $m$-cycle system of order $u$ can be embedded in an $m$-cycle system of order $v$ if and only if $u,v\equiv 1$ or $m \mod{2m}$ and $v\geq \frac{u(m+1)}{m-1}+1$.
\end{lemma}

\begin{proof}
Let $m=p^n$ for some odd prime $p$ and some integer $n\geq 1$. If there exists an $m$-cycle system of order $v$ containing a subsystem of order $u$ then $p^n$ divides $\binom{u}{2}=\frac{u(u-1)}{2}$ and $\binom{v}{2}=\frac{v(v-1)}{2}$. Since $p$ cannot divide both $u$ and $u-1$, $u\equiv 1$ or $m\mod{2m}$ and by a similar argument $v\equiv 1$ or $m\mod{2m}$.  Also note that $v\geq \frac{(m+1)u}{m-1}+1$ by Lemma \ref{NecConds}.

Conversely, suppose that $v\geq \frac{(m+1)u}{m-1}+1$ and $u,v\equiv 1$ or $m\mod{2m}$. If $u\equiv v\equiv m\mod{2m}$, then $u \geq m$, $v-u\geq 2m$ and the result follows by Theorem \ref{MainTheorem}. Otherwise either $u\equiv 1 \mod{2m}$ or $v\equiv 1 \mod{2m}$ and the result follows directly from Theorem \ref{DWExceptions}.
\end{proof}

\begin{proof}[{\bf Proof of Theorem \ref{DWTheorem}.}]
Part (i) follows directly from Theorem \ref{CSExist}, Corollary \ref{MissingCasesTheorem}(iii) and Lemma \ref{DWprimes}. Part (ii) follows from Theorem \ref{CSExist} and Corollary \ref{MissingCasesTheorem}(ii) (note that an $m$-cycle system of order one is trivially embedded in any $m$-cycle system and that any non-trivial $m$-cycle system has order at least $m$).
\end{proof}

Finally we shall prove Theorem \ref{SmallMSolution}.

\begin{proof}[{\bf Proof of Theorem \ref{SmallMSolution}.}]
By Corollary \ref{MissingCasesTheorem} it is sufficient to find $m$-cycle decompositions of $K_v-K_u$ when $(u,v)$ is $m$-admissible and either
\begin{itemize}
	\item
$u < m-2$ and $v \leq \nu_m(u)+m-1$; or
	\item
$m-2 \leq u \leq \frac{(m-1)(m-2)}{2}$ and $v\leq u+m-1$.
\end{itemize}
By Theorem \ref{DWExceptions} we also know that, for pairs $(u,v)$ such that $u\equiv 1 \mod{2m}$ and $v\equiv m\mod{2m}$, there exists an $m$-cycle decomposition of $K_v-K_u$ when $v\geq \frac{u(m+1)}{m-1}+1$. So simple calculation reveals that it suffices to find an $m$-cycle decomposition of $K_v-K_u$ for the values of $m$, $u$ and $v$ given in the following table.

\begin{center}
\begin{small}
\begin{tabular}{|c|p{15.5cm}|}
  \hline
 $m$ & $(u,v)$ \\
 \hline
 $9$ & $(5, 11)$, $(5, 17)$, $(11, 17)$, $(17, 23)$ \\
\hline
 $11$ & $(5, 27)$, $(5, 29)$, $(7, 27)$, $(7, 29)$, $(13, 21)$, $(25, 31)$, $(35, 43)$ \\
\hline
 $13$ & $(5, 31)$, $(5, 35)$, $(7, 33)$, $(9, 31)$, $(9, 35)$, $(15, 25)$, $(29, 37)$, $(41, 51)$ \\
\hline
 $15$ & $(5, 35)$, $(5, 41)$, $(7, 19)$, $(7, 27)$, $(9, 19)$, $(9, 27)$, $(11, 35)$, $(11, 41)$, $(15, 25)$, $(17, 29)$, $(21, 31)$, $(27, 37)$, $(27, 39)$, $(33, 43)$, $(37, 49)$, $(39, 49)$, $(45, 55)$, $(47, 59)$, $(49, 57)$, $(51, 61)$, $(57, 67)$, $(57, 69)$, $(63, 73)$, $(67, 79)$, $(77, 89)$ \\
\hline
\end{tabular}
\end{small}
\end{center}

We have found the desired decomposition in each of these cases using a computer program that implements basic cycle switching techniques to augment decompositions. These decompositions are included as supplementary material with arXiv:1411.3785.
\end{proof}

\vspace{0.3cm} \noindent{\bf Acknowledgements}

The first author was supported by Australian Research Council grants DE120100040, DP120103067 and DP150100506. The second author was supported by an Australian Postgraduate Award.

\bibliographystyle{plain}
\bibliography{references}
\end{document}